\documentclass[12pt,oneside,reqno]{amsart}
\usepackage{amssymb,amsfonts,amsmath}
\usepackage{epsfig}
\usepackage{graphicx}
\usepackage{color}
\usepackage{anysize}

\marginsize{2cm}{2cm}{1cm}{2cm}

\begin{document}
\baselineskip=14pt

\newcommand{\bfc}{\color{blue}} 
\newcommand{\gfc}{\color{red}}
\newcommand{\ec}{\color{black}}

\newtheorem{defin}{Definition}[section]
\newtheorem{teo}{Theorem}[section]
\newtheorem{ml}{Main Lemma}
\newtheorem{con}{Conjecture}
\newtheorem{cond}{Condition}
\newtheorem{conj}{Conjecture}
\newtheorem{prop}[teo]{Proposition}
\newtheorem{lem}{Lemma}[section]
\newtheorem{rmk}[teo]{Remark}
\newtheorem{cor}{Corollary}[section]
\renewcommand{\theequation}{\thesection .\arabic{equation}}
\numberwithin{equation}{section}

\newcommand{\beq}{\begin{equation}}
\newcommand{\eeq}{\end{equation}}
\newcommand{\beqn}{\begin{eqnarray}}
\newcommand{\beqnn}{\begin{eqnarray*}}
\newcommand{\eeqn}{\end{eqnarray}}
\newcommand{\eeqnn}{\end{eqnarray*}}
\newcommand{\bprop}{\begin{prop}}
\newcommand{\eprop}{\end{prop}}
\newcommand{\bteo}{\begin{teo}}
\newcommand{\bcor}{\begin{cor}}
\newcommand{\ecor}{\end{cor}}
\newcommand{\bcon}{\begin{con}}
\newcommand{\econ}{\end{con}}
\newcommand{\bcond}{\begin{cond}}
\newcommand{\econd}{\end{cond}}
\newcommand{\bconj}{\begin{conj}}
\newcommand{\econj}{\end{conj}}
\newcommand{\eteo}{\end{teo}}
\newcommand{\brm}{\begin{rmk}}
\newcommand{\erm}{\end{rmk}}
\newcommand{\blem}{\begin{lem}}
\newcommand{\elem}{\end{lem}}
\newcommand{\ben}{\begin{enumerate}}
\newcommand{\een}{\end{enumerate}}
\newcommand{\bei}{\begin{itemize}}
\newcommand{\eei}{\end{itemize}}
\newcommand{\bdf}{\begin{defin}}
\newcommand{\edf}{\end{defin}}
\newcommand{\bpr}{\begin{proof}}
\newcommand{\epr}{\end{proof}}

\newcommand{\halmos}{\rule{1ex}{1.4ex}}
\def \qed {{\hspace*{\fill}$\halmos$\medskip}}

\newcommand{\cda}{\overset{\alpha}{\cdot}}

\newcommand{\fr}{\frac}
\newcommand{\Z}{{\mathbb Z}}
\newcommand{\R}{{\mathbb R}}
\newcommand{\E}{{\mathbb E}}
\newcommand{\C}{{\mathbb C}}
\renewcommand{\P}{{\mathbb P}}
\newcommand{\N}{{\mathbb N}}
\newcommand{\var}{{\mathbb V}{\rm ar}}
\renewcommand{\S}{{\mathcal S}}
\newcommand{\T}{{\mathcal T}}
\newcommand{\W}{{\mathcal W}}
\newcommand{\X}{{\mathcal X}}
\newcommand{\Y}{{\mathcal Y}}
\newcommand{\h}{{\mathcal H}}
\newcommand{\f}{{\mathcal F}}
\newcommand{\ex}{{\mathcal E}}

\renewcommand{\a}{\alpha}
\renewcommand{\b}{\beta}
\newcommand{\g}{\gamma}
\newcommand{\G}{\Gamma}
\renewcommand{\L}{\Lambda}
\renewcommand{\l}{\lambda}
\renewcommand{\d}{\delta}
\newcommand{\D}{\Delta}
\newcommand{\e}{\epsilon}
\newcommand{\s}{\sigma}
\newcommand{\B}{{\mathcal B}}
\renewcommand{\o}{\omega}

\newcommand{\bj}{{\bf j}}

\newcommand{\nn}{\nonumber}
\renewcommand{\=}{&=&}
\renewcommand{\>}{&>&}
\newcommand{\<}{&<&}
\renewcommand{\le}{\leq}
\newcommand{\+}{&+&}

\newcommand{\pa}{\partial}
\newcommand{\ffrac}[2]{{\textstyle\frac{{#1}}{{#2}}}}
\newcommand{\dif}[1]{\ffrac{\partial}{\partial{#1}}}
\newcommand{\diff}[1]{\ffrac{\partial^2}{{\partial{#1}}^2}}
\newcommand{\difif}[2]{\ffrac{\partial^2}{\partial{#1}\partial{#2}}}

\parindent=0cm
\def\8{\infty}
\def\la{\lambda}
\newcommand{\eps}{\varepsilon}
\newcommand{\cvlaw}{\stackrel{\rm law}{\longrightarrow}}
\newcommand{\eqlaw}{\stackrel{\rm law}{=}}
\newcommand{\uu}{{\tt u}}
\newcommand{\Ta}{T^{(\alpha)}}
\newcommand{\phia}{\phi^{(\alpha)}}
\newcommand{\Sa}{{\mathcal S}_{\alpha}}
\newcommand{\Sd}{{\mathcal S}_{1/2}}
\newcommand{\xiij}{{ \xi_{i,j}}}
\newcommand{\omegaij}{{ \omega_{i,j}}}
\newcommand{\fxi}{{\mathfrak X}}
\def\cZ{\mathcal Z}

\newcommand{\VNAI}{V_{N,\alpha}^{\infty}}
\newcommand{\VNBI}{V_{N,\beta}^{\infty}}
\newcommand{\VINAI}{\overleftarrow{V}^{\infty}_{N, \alpha}}
\newcommand{\VNI}{V^{\infty}_{N}}
\newcommand{\VINI}{\overleftarrow{V}^{\infty}_{N}}

\title{
Random polymers on the complete graph }


\author{Francis Comets$^{1,5}$, Gregorio Moreno$^{2,4,5}$ and Alejandro F. Ram\'\i rez$^{3,4,5}$}

\thanks{ AMS 2000 {\it subject classifications}. Primary  60K35;
 secondary 82B23, 60B20
 }


\thanks{$^1$Universit\'e Paris Diderot. Partially supported by CNRS, Laboratoire de Probabilit\'es et Mod\`eles
 Al\'eatoires,
UMR 7599.}

\thanks{$^2$ Pontificia Universidad Cat\'olica de Chile. Partially supported by Fondecyt grant 1171257}

\thanks{$^3$ Pontificia Universidad Cat\'olica de Chile. Partially supported by Fondecyt grant 1141094}

\thanks{$^4$ Partially supported by N\'ucleo Milenio `Modelos Estoc\'asticos de Sistemas Complejos y Desordenados'}

\thanks{$^5$ Partially supported by Mathamsud `Large-scale Behavior of Stochastic Systems'}

\address[Francis Comets]{Universit\'e Paris Diderot - Paris 7\\
Math\'ematiques, case 7012\\
F-75 205 Paris Cedex 13, France}
\email{comets@math.univ-paris-diderot.fr}

\address[Gregorio Moreno, Alejandro F. Ram\'\i rez]{Facultad de Matem\'aticas\\
Pontificia Universidad Cat\'olica de Chile\\
Vicu\~na Mackenna 4860, Macul\\
Santiago, Chile}

\email{grmoreno@mat.uc.cl, aramirez@mat.uc.cl}
\bigskip

\begin{abstract} 
Consider directed polymers in a random environment on the complete graph of size $N$.
This model can be formulated as a product of i.i.d. $N\times N$ random matrices and its large time asymptotics
is captured by Lyapunov exponents and the Furstenberg measure. We detail this correspondence, derive the long-time limit of the model and obtain a co-variant distribution for the polymer path. 

Next, we observe that the model becomes exactly solvable when the disorder variables are located on edges of the complete graph and follow a totally asymmetric stable law of index $\alpha \in (0,1)$. Then, a certain notion of mean height of the polymer behaves like a random walk and we show that the 
height function is distributed around this mean according to an explicit law. Large $N$ asymptotics can be taken in this setting, for instance, for the free energy of the system and for the invariant law of the polymer height with a shift. Moreover, we give some perturbative results for environments which are close to the totally asymmetric stable laws. 

\end{abstract}


\date{November 23, 2017}
\maketitle
\bigskip

\section{Introduction and results}

Directed polymers in random environments were introduced in \cite{Henley-Huse} as a model for the phase separation line of the $2d$ Ising model in the presence of impurities. Since then, they have been the subject of an important body of work both in the mathematics and physics communities. On the mathematics side, their study started with the work \cite{Imbrie-Spencer}, shortly followed by \cite{Bolthausen}. An up-to-date account on the mathematical treatment of directed polymers can be found in \cite{SF}. On the physics side, this model was often studied in connection with growing surfaces. In particular, it was observed that the Kardar-Parisi-Zhang equation \cite{KPZ} can be considered as a continuum version of directed polymers \cite{KZ}.  
\medskip

In its usual discrete version, the model can be formulated in terms of two basic ingredients:
(i) {The polymer paths} are given a-priori by the trajectories of a simple symmetric random walk on $\Z^d$ starting at the origin, whose law we denote by $P$;
(ii) {The environment} is given by a family $\{\eta(t,x):\, t\geq 1,\, x\in \Z^d\}$ of i.i.d. random variables with common distribution $\P$.
Then, given $t\geq 0$, $\beta>0$ and a fixed realization of the environment, we define the polymer measure on nearest-neighbor paths $\bf x$ of length 
$t$ as the probability measure
\begin{eqnarray} \nn
	\mu^{\eta}_{t,\beta}(d{\bf x}) = \frac{1}{Z_{t,\beta}^{\eta}} \exp\{\beta \sum^t_{s=1}\eta(s,{\bf x}_s)\}P(d{\bf x}),
\end{eqnarray}
where $Z^{\eta}_{t,\beta}$ is the normalizing constant.

\vspace{1ex}

In this note, we propose to replace,  in the above definition, $\Z^d$
 with its natural graph structure 
 by the complete graph with $N$ sites. The precise definition of the model is deferred Section \ref{sec:intro}. Alternatively, this can be seen as a polymer with long-range jumps on a finite state space.
 
 \vspace{1ex}

Large finite state spaces have been considered in the literature as approximations of infinite systems, for instance \cite{BrunetDerrida00} for directed polymers on a cylinder, or directed polymers on $n$-trees as in \cite{Cook-Derrida}. In \cite{BrunetDerrida04}, a zero-temperature version of our model was introduced to compute the corrections for large finite systems to continuous limit equations of front propagation.

\vspace{1ex}

There are several additional reasons to consider polymers on other  graphs. First of all, they offer several simplifications compared to the original model on $\Z^d$ or $\R^d$. In this respect, models on the complete graph  can be seen as a mean field approximation, much in the same way the Curie-Weiss model relates to the Ising model. When considering mean-field versions of directed polymers it is more common to refer to polymers on the tree \cite{Buffet-Pule} \cite{DerridaSpohn}, but these models show fundamental differences with the original ones \cite{SF}. 
Our model on the complete graph preserves some features, and can also be seen as a positive temperature version of the last passage percolation studied in \cite{BrunetDerrida04, CQR} from where we draw many ideas.

\vspace{1ex}

This approximation of infinite graphs by finite ones can be made quantitative. For instance, the paper \cite{Eckmann89} considers a particular product of random symplectic matrices corresponding to a random diffusion on the $d$-dimensional discrete torus. Assuming $d \geq 3$ and weak  disorder, it is shown that, as the sidelength of the torus goes to infinity, 
the largest Lyapunov exponent of the product converges to the one of the usual Laplacian on the full lattice.

\vspace{1ex}

In a related spirit, models on arbitrary countable graphs are considered in \cite{CaGuHuMe} where the {\it a priori} law on the path space is an ergodic Markov chain.

\vspace{1ex}

Finally, on the complete graph, we are able to find a law on the environment that makes the model solvable in the sense that the law of the properly normalized partition function can be computed explicitly. This is analogous to the zero-temperature version of the model from \cite{BrunetDerrida04,CQR,Cortines}.
Exactly solvable models are rare in statistical mechanics, but  they yield most informative  results. This one seems to be new. In the case of polymers on the lattice, we can however cite \cite{Timo} for the original model on $\Z$ and \cite{OCY} for a semi-continuous counterpart. 

\vspace{1ex}

The study of polymer models on finite graphs falls in the scope of products of random matrices. More precisely, the free energy is related to the top Lyapunov exponent of such products for which a complete theory has been established in the literature initiated in  \cite{FurstKest60} and developped  in particular in \cite{BouLa85,CN84,Furst63,GuivRaugi85,Hennion97,HennionHerve08,Newman86}. 
This formalism allows us to derive, for example, the existence of the free energy as well as Gaussian fluctuations for the logarithm of the partition function. 
In different direction, we mention \cite{Arnold} which gives a nice account on Perron-Frobenius theory for product of random matrices together with its relation 
to thermodynamic formalism. 
For "infinite matrices" -- i.e. for random positive operators in infinite dimension, limits  may fail to exist due to lack of compactness. The authors in \cite{Bakhtin-Khanin, Sinai}  deal with products of random operators on the whole $\Z^d$, but some localization features -- a priori embedded in the special models  --
make them essentially compact.

\subsection{Directed polymers on the complete graph}\label{sec:intro}
We study the following model of directed polymers on the complete graph with $N$ sites: for $t\geq 1$ integer and $1\leq i, j \leq N$, consider the set of paths  starting at location $i$ at time 0 and ending at $j$ at time $t$,
\begin{eqnarray}\label{eq:path-space}
	J_N(0,i; t,j)=\big\{\bj=(j_0,\cdots,j_t):\, 1\leq j_s \leq N,\, \forall_{1\leq s \leq t-1}; \, j_0=i, \, j_t=j\big\}.
\end{eqnarray}
Let $\{\omega_{i,j}(t):\, 1\leq i,j \leq N,\, t\geq 0\}$ be a family of i.i.d. positive random variables defined on some probability space $(\Omega, {\mathcal A}, \P)$. For a fixed realization of the environment 
and $t\geq 1$, we define the {\em point-to-point} (P2P) polymer partition function
\begin{eqnarray}\label{eq:pf-P2P}
	\cZ_N(0,i;t,j)=\sum_{{\bf j} \in J_N(0,i;t,j)} \prod^t_{s=1}\omega_{j_{s-1},j_s}(s) .
\end{eqnarray}
Viewed as a function of $j$, its logarithm is referred to as the polymer {\em height function}. 
We denote with a $\star$ quantities  $Q(\star)$ which are, depending on the context, unions or sums over locations $j \in \{1,\ldots N\}$ of $Q(j)$.
For instance,
$J_N(0,i;t,\star)=\cup^N_{j=1}J_N(0,i;t,j)$, and $\cZ_N(0,i;t,\star)= \sum_{j=1}^N \cZ_N(0,i;t,j)$ is the so-called {\em point-to-line} 
(P2L) polymer partition function. Similarly, $J_N(0,\star;t,j)=\cup^N_{i=1}J_N(0,i;t,j)$, and $\cZ_N(0,\star;t,j)= \sum_{i=1}^N \cZ_N(0,i;t,j)$ is the 
 {\em line-to-point} (L2P) partition function. 
\medskip

Our analysis will rely on a tight relation between our model and products of random matrices which emerges from the following observation:
defining the  $N \times N$ matrix $\Pi(t)$ as the product
\begin{equation}\nn
\Pi(t)=\fxi(1)\fxi(2)\cdots\fxi(t),
\end{equation}
of the matrices $\fxi(t)=[\omega_{i,j}(t)]_{i,j}$, we see that the P2P partition function \eqref{eq:pf-P2P} is the $(i,j)$-entry of $\Pi(t)$, 
\begin{equation}
\label{eq:P2P=Pi}
\cZ_N(0,i;t,j)= \Pi(t)_{i,j} \;.
\end{equation}
(We do not indicate in the notation the dependence in $N$ of $\fxi$ and $ \Pi$, although we will let $N$ go to infinity at some stage.)
Let us denote by $Z_N(t)$ the (column) 
vector given by the L2P partition functions 
$$ Z_N(t)= \big(\cZ_N(0,\star; t,1),\cdots,\cZ_N(0,\star; t,N)\big)^*$$ 
where $M^*$ denotes the transposed of the matrix $M$. Our convention in the paper is that all vectors are column vectors, and we write $v(j)$ for the $j$-th coordinate of $v$, so that $Z_N(t,j)\equiv \cZ_N(0,\star; t,j)$. From \eqref{eq:P2P=Pi}, we obtain
\begin{eqnarray} \label{eq:recursion-basic-3}
Z_N(t)^* = {\bf 1}^* \Pi(t),
\end{eqnarray}
where ${\bf 1}^*$ denotes the $N$-dimensional row vector with all entries equal to $1$. Similarly, the vector of P2L partition functions can be written as 
$\Pi(t) {\bf 1}$.
\smallskip

This point of view allows us, among other things, to relate the free energy of the model to the  Lyapunov exponent of products of i.i.d. random matrices. 
In our case the matrices have an additional feature -- entries are i.i.d. -- but the theory applies to the general case under mild assumptions. 
At this point, it is convenient to take \eqref{eq:recursion-basic-3} as the starting point of our analysis and consider the slightly more general framework of the recursion
\begin{eqnarray}\label{eq:recursion-general}
Z_N(t)^* 
= Z_N(t-1)^*\fxi(t)
\end{eqnarray}
allowing general initial conditions $Z_N(0)\in \R^N_+ \setminus \{0\}$..
\vspace{1ex}

\medskip


We follow the formalism of \cite{Hennion97, HennionHerve08} based on the action of products of random matrices on projective spaces. For $v\in \R^N_+ \setminus \{0\}$ and $\alpha>0$, define the $\alpha$-norm of $v$ as $||v||_{\alpha}=(\sum^N_{j=1}v_i^{\alpha})^{1/\alpha}$. Of course, this quantity is a norm only in the case $\alpha\geq 1$. Next, we introduce the $\alpha$-symplex
$$\bar B_{\alpha}=\{v\in \R^N_+:\, ||v||_{\alpha}=1\},$$
together with the projection $\Psi_{\alpha}(v)=\frac{v}{||v||_{\alpha}}$ from $\R_+^N \setminus \{0\}$ onto $\bar B_\alpha$.

For  $v\in \R^N_+ \setminus \{0\}$ and $\fxi$ an $N$-by-$N$ matrix with positive entries, we define the product $\cda $ by
$$\fxi \cda v := \frac{\fxi v}{||\fxi v||_{\alpha}}\in \bar{B}_{\alpha}.$$
We will drop the subscripts and superscripts from the notation when $\alpha=1$ and write
$$\bar B := \bar B_1,\quad \fxi \cdot v:= \fxi \overset{1}{\cdot} v.$$
Finally, define
\begin{eqnarray} \label{def:XNA}
	X_{N,\alpha}(t) := \Psi_{\alpha}(Z_N(t))=\frac{Z_N(t)}{||Z_N(t)||_{\alpha}} \in \bar B_{\alpha},
\end{eqnarray}
and, again, write $X_N:=X_{N,1}$. This leads to the simple decomposition
\begin{eqnarray}\label{eq:decomposition}
	\log Z_{N}(t,i) = \log ||Z_N(t)||_{\alpha}+\log X_{N,\alpha}(t,i)
\end{eqnarray}
Note that, by the recursion \eqref{eq:recursion-general} and homogeneity, we have
\begin{eqnarray*}
	X_{N,\alpha}(t) = \Psi_{\alpha}\big( \fxi(t)^*Z_{N,\alpha}(t-1) \big) = \Psi_\alpha\big(\fxi(t)^*X_{N,\alpha}(t-1)\big) ,
\end{eqnarray*}
showing that $\{X_{N,\alpha}(t):\, t\geq 0\}$ is a Markov chain. We list further important properties of this chain in the next theorem. 

\subsection{Product of random matrices and polymer model structure}

{ All three results in this subsection come as applications of the general theory of product of independent random matrices. 
They are overlooked in this context, although they provide a complete understanding of the model for a fixed $N$. }
For  integers  $s < t$, let
 \begin{equation*}
\Pi(s,t)= \fxi(s+1) \ldots \fxi(t), \qquad \Pi(t,t)={\rm I}_N, \qquad \Pi(t)=\Pi(0,t).
\end{equation*}

\begin{teo}\label{thm:structure}
Let $Z_N(0)\in \R^N_+$.
\begin{enumerate} 
	\item For all $\alpha>0$, the recursion \eqref{eq:recursion-general} with initial condition $Z_N(0)$ defines by \eqref{def:XNA}
	 a time-homogeneous Markov chain $\{X_{N,\alpha}(t):\, t\geq 0\}$ with values in $\bar B_{\alpha}$.
	
	\item There exists an event $\Omega_0$ with $\P(\Omega_0)=1$  such that the (random) limit 
			\begin{eqnarray}\label{eq:Xinf}
				\VNAI= \lim_{t\to\infty} \Pi(t) \cda v,
			\end{eqnarray}
			exists for all $\alpha>0, \omega \in \Omega_0$ and does not depend on $v\in \R^N_+$. Moreover, $\VNAI= \Psi_\alpha(\VNBI)$ for all $\alpha, \beta>0$.
		
	\item Let  $m_{N,\alpha}$ denote the law of  $\VNAI$.
	The chain $(X_{N,\alpha}(t))_{t\geq 0}$ with initial law $m_{N,\alpha}$ is stationary and ergodic.
	
	\item Denote by $\theta_s$ the shift on $\Omega$ by $s \in \Z$,  $\theta_s \omega (t) = \omega (s+t)$, and set 
	\begin{eqnarray}\label{eq:Xinfshift}
				\VNAI(s) := \VNAI\circ \theta_{s} = \lim_{t\to\infty} \Pi(s,t) \cda v ,
			\end{eqnarray}
	and $\VNAI(s,j)$ the $j$-th component of this vector.
	(In particular,  $\VNAI(0)=\VNAI$.) Then,
\begin{eqnarray}\label{eq:Xinfinvariant}
	\fxi(0) \cda \VNAI= \VNAI(-1)
\end{eqnarray}

\end{enumerate}
\end{teo}
This is proved in Section \ref{sec:2.1} and the Appendix.
\vspace{1ex}

Our next result states the almost sure existence of the free energy as well as Gaussian fluctuations for the logarithm of the partition function. From \eqref{eq:decomposition}, note that the asymptotics of $\log Z_N(t,i)$ is essentially given by that of the first term. That this limit is independent of $\alpha$ comes from the observation that, for each $\alpha>0$, there exists a constant $c_N(\alpha)\in(1,\infty)$ such that
$$c_N(\alpha)^{-1} ||v||_{1} \leq ||v||_{\alpha} \leq c_N(\alpha)\,||v||_{1}, \quad \forall \, v\in \R^N_+.$$
\begin{teo} \label{thm:gene-intro} Fix $N$, assume that the $\omega$'s are not constant and that $\E |\log \omega_{i,j}|^{2+\delta}<\8$ for some positive $\delta$. Then, there exist numbers $v_N$ and $\sigma_N >0 $
such that, for all $j=1,\ldots N$, 
$$
\lim_{t \to \8} \frac{1}{t} \log Z_N(t,j) = v_N \qquad {\rm a.s.},
$$
and
$$
\frac{1}{\sqrt t} \big( \log Z_N(t,j) - v_N t \big) \cvlaw {\mathcal N} (0, \sigma_N^2) \qquad {\rm as}\,\, t\to\infty.
$$
Furthermore, with $V^{\infty}_{N,\alpha}$ from \eqref{eq:Xinf}, 
\begin{eqnarray} \label{eq:vN=}
	v_N = \E\left[ \log || \fxi(0) \VNAI||_{\alpha}\right].
\end{eqnarray}
\end{teo}
We give a proof in Section \ref{sec:FELE}. 
\vspace{1ex}

We now turn to the asymptotic of the polymer measure. The P2L polymer measure starting at $i$ with time-horizon $T$, denoted by 
$P^\omega_{0,i;T,\star}$, is the (random) probability measure on $J_N(0, i; T, \star)$ given by
\begin{eqnarray}\label{eq:polymer-measure}
	P^\omega_{0,i;T,\star}\big(\, \bj = (j_0,\cdots, j_T)\big) =
	\frac{ 1}{Z_N(0,i;T,\star)} \;\prod^T_{s=1}\omega_{j_{s-1},j_s}(s)\;.
\end{eqnarray}
Similarly to Theorem \ref{thm:structure}, there exists an almost-sure limit to the ``backwards-in-time" product
\begin{equation} \label{eq:Xinfreverse}
\VINAI(s) 	 = \lim_{t\to -\infty} \Pi(t,s)^* \cda v
\end{equation}
which does not depend on $v \in \R_+^N$.  Since $\fxi(0)^* \eqlaw \fxi(0)$, we have
\begin{equation} \label{eq:Xinfreverselaw}
\VINAI (s)	 \eqlaw \VNAI .
\end{equation}
Our second set of results states the existence of an infinite volume polymer measure together with a co-variant law: define the random probability measure $\nu_N (t,\cdot)$ on 
$\{1,\ldots N\}$ by
\begin{equation} \label{def:co-var}
\nu_N (t, j):= \frac{\VINAI(t,j) \VNAI(t,j) }{ \sum_{k=1}^N \VINAI(t,k) \VNAI(t,k)} = \frac{\VINI(t,j) \VNI(t,j) }{ \sum_{k=1}^N \VINI(t,k) \VNI(t,k)}\; ,
\end{equation}
since the ratio in the second term does not depend on $\alpha$, by Theorem \ref{thm:structure}, point 2.
In words, the co-variant law is proportional to the doubly infinite product of weights over polymers (from times $-\8$ to $+\8$) which take the value $j$ at time $t$.
\begin{teo} \label{prop:infinitepolymer}
\begin{enumerate}
\item \label{it:(i)}For almost every environment  $\omega$, the polymer measure $P^\omega_{0,i;T,\star}$ 
converges as $T \to \8$  to the (time-inhomogeneous) Markov chain with $P^{\omega}(j_0=i)=1$ and transition probabilities given by
\begin{eqnarray} \label{eq:polymerinfini}
P^\omega( j_{t+1}=\ell \big \vert j_t=k) &=&
\frac{\omega_{k,\ell}(t\!+\!1)\VNI(t\!+\!1,\ell) }{\sum^N_{\ell'=1}\omega_{k,\ell'}(t\!+\!1)\VNI(t\!+\!1,\ell')} \\
\label{eq:polymerinfini2}
&=& \frac{1}{\| \fxi(t\!+\!1) \VNI(t\!+\!1) \|_1} \times
\frac{\omega_{k,\ell}(t\!+\!1)\VNI(t\!+\!1,\ell) }{\VNI(t,k) }
\end{eqnarray}
 for $t\geq 0, k, \ell \in \{1,\ldots N\}$. 
\item \label{it:(ii)}
Let $\omega \in \Omega_0$. For the chain with transition \eqref{eq:polymerinfini} starting at time $s$ with law $\nu_N(s,\cdot)$, we have for $t \geq s$,
 $$
 P^\omega( j_{t}=\ell ) = \nu_N(t, \ell)\;, \qquad \ell =1,\ldots N.
 $$ 
\end{enumerate}
\end{teo}
This is proved in Section \ref{sec:infvolmeas}.
\vspace{1ex}

Note that, if $Z_N(0,\cdot;t,\star)$ denotes the vector $(\mathcal{Z}_N(0,1;t,\star),\cdots, \mathcal{Z}_N(0,N;t,\star))^*$, then
\begin{eqnarray}
	Z_N(0,\cdot;t;\star) = \Pi(t) \, {\bf 1},
\end{eqnarray} 
where $\Pi(t)$ is defined in \eqref{eq:recursion-basic-3}. This representation allows us to define a polymer measure with more general initial conditions. The above theorem can be easily adapted to this setting.

\begin{rmk} 
The law $m_{N}=m_{N,1}$ is the long-time limit of the endpoint distribution, which appears as a fixed point of a transfer operator on $\Z^d$ for the nearest neighbor graph in \cite{Bates-Chatterjee}, and for long range graphs in \cite{Bates}.
(Of course, for a finite state space, there is no question of localization and the disorder is always strong.)
The previous theorems provide much more information in this simplified framework.
\end{rmk}

\subsection{Case of $\alpha$-Stable Environments}

We consider now the particular cases when  the environment follows a stable law of index $\alpha \in (0,1)$, see \cite{BertoinS, Durrett}.  The law $\Sa$ is supported on $\R_+$ and can be defined via its Laplace transform: if $S$ is distributed according to $\Sa$, then
\begin{equation}
\E e^{-\la S}= e^{- \la^{\alpha}}, 
\end{equation}
for all $\lambda \geq 0$.
In particular, if $S_1,\cdots,S_N$ are $N$ independent $\Sa$-distributed random variables, then
\begin{equation} \label{eq:stabilite}
N^{-1/\alpha} \sum_{i=1}^N S_i \eqlaw \Sa,
\end{equation}
and, more generally,
\begin{eqnarray}\label{eq:stabilite2}
	\sum^N_{i=1} a_i S_i \eqlaw \Sa,
\end{eqnarray}
provided $\sum^N_{i=1}a_i^{\alpha}=1$ and $a_i \geq 0$. The tail of $\Sa$ is known to decay polynomially,
\begin{eqnarray*}
	\P[S > x] \sim \frac{1}{\Gamma(1-\alpha)}x^{-\alpha},
\end{eqnarray*}
as $x \to \8$.
Furthermore, $\Sa$ is the limit of properly normalized sums of i.i.d. random variables exhibiting similar decay. 
\vspace{1ex}

It turns out that this choice of environment makes the model solvable, in the sense that the law of the (properly normalized) partition function is explicit. 
The decomposition \eqref{eq:decomposition} has to be slightly modified to reveal the rich structure of this version of the model: let
\begin{equation} \label{def:Phi_N}
	S_N(t,j):=\frac{Z_N(t,j)}{|| Z_N(t-1)||_{\alpha}},\quad \phi_N(t):=\log || Z_N(t)||_{\alpha},
\end{equation}
so that
\begin{eqnarray} \label{eq:indepsum}
	\log Z_N(t,j) = \log S_N(t,j) + \phi_N(t-1).
\end{eqnarray}
There are many reasonable manners to measure the mean height of the polymer. However the $\alpha$-norm yields an unexpectedly simple description.
The full probabilistic structure of the stable case is detailed in the next theorem.
\begin{teo}\label{thm:structure-stable}
	Suppose $\{\omega_{i,j}(t):\, t\geq 1,\, 1\leq i,j \leq N\}$ is an i.i.d. family of $\Sa$-distributed random variables. Then,
	\begin{enumerate}
		\item $\{S_N(t,j):\, t\geq 1,\, 1\leq j \leq N\}$ is an i.i.d. family of $\Sa$-distributed random variables. Moreover, the terms in the sum \eqref{eq:indepsum} are independent. 
		\item Starting from any state, the Markov chain $X_{N,\alpha}(\cdot)$ reaches equilibrium instantaneously. In fact, 
		$(X_{N,\alpha}(t))_{t \geq 1}$ is an i.i.d. sequence in $\bar B_\alpha$ for all starting point $X_{N,\alpha}(0)$.
		\item $\{\phi_N(t):\, t\geq 1\}$ is a random walk with i.i.d jumps $\{\Upsilon_N(t):\, t\geq 1\}$ distributed as
		\begin{equation} \label{eq:accr-stable}
\Upsilon_N \eqlaw  \log \| S_N\|_{\alpha},
\end{equation}
where $S_N$ is 
a $N$-vector with i.i.d.~$\Sa$-distributed coordinates.
	\item $v_N=\E[\Upsilon_N]$, $\sigma^2_N = \mathbb{V}ar[\Upsilon_N]$.
	
	\item The invariant law $m_{N,\alpha}$ has the same distribution as $\frac{S_N}{|| S_N||_\alpha}$. 
	
	\item The sequence 
	$\{ (\VNAI(t,j))_{j=1}^N\}_{ t\geq 1}$   is i.i.d. with common distribution $m_{N,\alpha}$.
	\end{enumerate}
\end{teo}

This is proved in Section \ref{sec:RWrepsentation}.

Note that, from the product of random matrices point of view, our results look very close to \cite{CN84} in spirit, although only the symmetric stable case is treated there. 
 However, an inspection of their proofs  shows that they do not cover the case of positive totally asymmetric stable laws studied here.

As the velocity and variance from Theorem \ref{thm:gene-intro} are now explicit, we can try to obtain their asymptotics when $N$ grows. 
\begin{prop}\label{thm:main-stable} 
Assume $\{\omega_{i,j}(t):\, t\geq 1,\, 1\leq i,j \leq N\}$ is an i.i.d. family of $\Sa$-distributed random variables.
Let 
\begin{equation} \label{eq:calpha}
c_{\alpha} := \Gamma(\alpha) \frac{\sin{\pi \alpha}}{\pi \alpha} .
\end{equation}
 Then, as $N \to \8$, 
\begin{eqnarray} \label{eq:val-v_N}
v_N &=&\alpha^{-1} \big( \log N + \log \log N + \log c_{\alpha}\big)+ o(1), \\ \nn
\sigma^2_N &=& \frac{\pi^2}{3 \alpha^2  \log N} + o(\frac{1}{\log N}). 
\end{eqnarray}
\end{prop}
This is proved in Section \ref{sec:asymptLyap}.

We now state the convergence of the rescaled random walk or polymer height to a L\'evy process:
\begin{teo}\label{thm:main-front}
Assume $\{\omega_{i,j}(t):\, t\geq 1,\, 1\leq i,j \leq N\}$ is an i.i.d. family of $\Sa$-distributed random variables.
Then, for any sequence $k_N\to \infty$, we have that 
\begin{eqnarray} \label{eq:ouf}
	\frac{\phi_N({k_N\tau})-\gamma_N k_N \tau}{k_N/\log N^{\alpha}} &\to& \S(\tau),
\end{eqnarray}
in law in the Skorohod topology, 
where $\phi_N$ is defined in \eqref{def:Phi_N} and 
\begin{eqnarray*}
	\gamma_N &=&  \frac{1}{\alpha}\log \left( \frac{N \log N}{\Gamma(1-\alpha)}\right)+\frac{\log k_N}{\alpha \log N} 
\end{eqnarray*}
and
 $\S(\cdot)$ is a totally asymmetric L\'evy process with exponent
\begin{eqnarray}
	\psi(u) = \int^{\infty}_1(e^{iux}-1)\frac{dx}{x^2}+\int^1_0(e^{iux}-1-iux)\frac{dx}{x^2}.
\end{eqnarray}
\end{teo}
The proof is very close to the one of the corresponding statement in \cite{CQR} and is given  in Section \ref{sec:poly-height}.

Finally, we obtain a Poisson-type convergence result for the invariant measure $m_{N,\alpha}$ in the case of an $\Sa$-distributed environment. This is presented in Section \ref{sec:inv-measure}.


\subsection{Perturbative results}
We will now study the case of environments that are perturbations of the $\Sa$ laws. 

Let $\alpha\in(0,1)$ and suppose $\{\omega_{ij}(t):\, t\geq 1,\, 1\leq i,j \leq N\}$ is an i.i.d. family of random variables with a common Laplace transform 
\begin{equation} \nn
\varphi(u)= \E\big[ \exp \big\{-u \; \omega_{i,j} (t)\big\}\big] ,\qquad  u \geq 0,
\end{equation}
 such that
\begin{equation} \label{eq:closetoSa}
	1 - \varphi(u) \sim u^{\alpha},\quad 
 u \to 0^+.
\end{equation}
We view such an environment as a perturbation of the $\alpha$-stable distributed environment, since $\omega$ lies in the domain of attraction 
of the $\alpha$-stable law.
It is important to note that $\E \omega_{i,j}(t)=\8$, therefore the 
various partition functions are not integrable and cannot be normalized. In particular, the models we consider are
outside the range of application of the main techniques in the field of directed polymers \cite{SF}.
Let  $u_{\alpha}$ denote the distribution function of the logarithm of an $\Sa$ random variable
\begin{equation} \label{def:ualpha}
u_{\alpha} (x) = \P ( \Sa > e^x),\qquad x \in \R, 
\end{equation}
 and let 
$U_N$ denote the {\em front profile} of the polymer,	
\begin{equation} \label{def:front}
U_N(t,x):=\frac{1}{N}\sum^N_{j=1}{\bf 1}_{\log Z_N(t,j)>x} \;, \qquad t \in \N, x \in \R.
\end{equation}
The random function $x \mapsto U_N(t,x)$ is the (inverse) distribution function of the polymer height function. 
\begin{teo} \label{th:perturb}
	Suppose the $\omega$'s satisfy \eqref{eq:closetoSa} for a given $\alpha \in (0,1)$.
	\begin{enumerate}
\item 
	Then, for all $t\geq 2$ and any  $i\in \{1,\cdots,N\}$, we have
	\begin{eqnarray*}
		\frac{Z_N(t,i)}{||Z_N(t-1)||_{\alpha}} \cvlaw \Sa.
	\end{eqnarray*}
	Moreover, for each $k\geq 1$ and any sequence $K_N\subset \{1,\cdots,N\}$ with $|K_N|=k$, we have
	\begin{eqnarray*}
		\left\{\frac{Z_N(t,i)}{||Z_N(t-1)||_{\alpha}}:\, i \in K_N\right\} \cvlaw \Sa^{\otimes k}.
	\end{eqnarray*}
\item
	Furthermore,  for all $t\geq 2$, we have, with $\phi_N$ from \eqref{def:Phi_N} and $u_\alpha$ from \eqref{def:ualpha}
			$$U_N\big(t, x + \phi_N(t-1)\big) \to u_{\alpha}(x) \quad a.s.,$$
			as $N\to \infty$, uniformly in $x$.
			\end{enumerate}
\end{teo}
{\bf Comment}: Since the polymer is pulled by the largest values, the height function is expected to grow like a moving front. Roughly,  the front behaves like a wave  traveling at speed $v_N$, 
and one could try to look at it at time $t$ around location $v_Nt$. But if the cardinality $N$ of the monomer state space is large, the actual 
front can be rather different from $v_Nt$, so it is natural  to look at it relative to a suitable random location $\phi_N(t-1)$. Item (2) in Theorem     
\ref{th:perturb} shows that, for $\omega$ close to an $\alpha$-stable law (i.e., an exactly solvable model), the polymer height function seen from this  location $\phi_N(t-1)$ converges -- without scaling -- to the exactly solvable model. 

\medskip

Theorem \ref{th:perturb} is proven in Section \ref{sec:proof-perturb}.
In Section \ref{sec:frontSa} we will also prove more complete results in the case of stable environments, including
a fluctuation theorem; see Proposition \ref{prop:frontprofilestable}.



\section{General environments with fixed $N$}
We note once for all that $\Psi_{\alpha}$ is a continuous bijection from $\bar{B}\equiv \bar{B}_1$ to $\bar{B}_{\alpha}$ for any $\alpha>0$. Hence, it is enough to work with $\alpha=1$. The objects constructed on $\bar{B}_1$ can then be projected on the other $\alpha$-symplexes yielding the results for general values of $\alpha$. In particular, the relation $\VNAI= \Psi_\alpha( \VNBI)$ appearing in part 2 of Theorem \ref{thm:structure},
follows immediately.
Hence, we will restrict to $\alpha=1$ for the rest of this section.

 We follow the well-known theory of product of random matrices, see e.g. \cite{Hennion97, HennionHerve08}. 

\subsection{Asymptotics of the Markov chain and stochastic contractivity} \label{sec:2.1}

In this section we prove Theorem \ref{thm:structure}, using  the following lemma, cf. section I in \cite{Hennion97}, which is the involved step. We reproduce the proof in the appendix as it contains some crucial ideas, in particular a contraction property. (And, moreover, it is beautiful.)

\begin{lem} \label{th:cvdir}
There exists a r.v. $\VNI$ taking values in $B$ such that, for all $v \in \bar B$,  $\Pi(t) \cdot v$ converges a.s. to 
$\VNI$ as $t \to \8$.  The convergence is a.s. uniform in $v \in \bar B$. 
\end{lem}
This immediately implies the almost sure convergence stated in part $2.$ of Theorem \ref{thm:structure}.
Note that we can conclude that $X_N(t)$ converges in law to $m_N$ as $t\to\infty$.

\begin{lem} \label{th:cvdir2}
The law $m_N$ of $\VNI$ is the unique 
invariant probability, i.e., the unique probability measure on $\bar B$ such that, for all 
bounded continuous $f: \bar B \to \R$,  
$$
\int_{\bar B}  \E \left[f\big( \fxi \cdot v\big)\right]  dm_N(v) = \int_{\bar B}   f( v)  dm_N(v) .
$$
\end{lem}
The law $m_N$ is usually called the Furstenberg measure. 
\begin{proof}
 From Lemma \ref{th:cvdir},  $V'=\lim_t \Pi(2,t) \cdot x$ converges a.s., and has the 
same law $m_N$. The equality $\fxi(1) \cdot V'= \VNI$ is the claimed  invariance property. Moreover, if $m_N'$ is another invariant law, we get by iterating $t$ times,
$$
\int_{\bar B}  \E \left[f\big( \Pi(t) \cdot x\big)\right]  dm_N'(x) = \int_{\bar B}   f( x)  dm_N'(x) .
$$
By dominated convergence, the left-hand side converges to $\E f(\VNI)$, and we can conclude that $m'$ is the law of $\mu_{\infty}$. 
\end{proof}

\begin{lem}
	 The law $m_N$ is invariant for the Markov chain $(X_N(t); t \geq 0)$, and the chain $X_N$ with initial law $m_N$ is ergodic. 
	 For any  bounded  continuous $f:\bar B\to \R$ and any initial condition $X_N(0)\in \bar B$,
			\begin{equation}\label{eq:ergodic-sums}
				\lim_{t\to\infty} \frac{1}{t} \sum^t_{s=1} f\big(X_{N}(s)\big)=\E\big[ f(\VNI)\big],
				\quad \P {\rm -a.s.}
				\end{equation}
\end{lem}
\begin{proof}
We can write \eqref{eq:Xinfinvariant} as $\fxi(0) \cda \VNAI= \VNAI \circ \theta_{-1}$. Transposing this identity and using that
$(\fxi(0), \VNI,\VNI \circ \theta_{-1}) \eqlaw (\fxi(1), \VINI,\VINI \circ \theta_{1})$, we see that 
\begin{equation*}
\VINI \cdot \fxi(1) = \VINI \circ \theta_1 .
\end{equation*}
Since $\VINI $ and $\fxi(1) $ are independent and $\VINI$ is stationary with law $m_N$, this equality implies that $m_N$ is invariant. 

Ergodicity is shown in Lemma 3.3 in \cite{Hennion97}. Finally, the pointwise ergodic theorem \eqref{eq:ergodic-sums} follows 
also from  the previous 
and the contraction property, see Appendix \ref{app:contr}.
\end{proof}


\subsection{Free energy and Lyapunov exponents} \label{sec:FELE}
The Perron-Frobenius eigenvalue of the (strictly) positive matrix $\Pi(t)$ is the r.v.
\begin{equation} \nn
\lambda_N^{\rm PF}(t)= \min_{x \in (\R_+^*)^N} \max_{1 \leq i \leq N} \frac{\big(\Pi (t) x\big)_i}{x_i} = 
\max_{x \in (\R_+^*)^N} \min_{1 \leq i \leq N} \frac{\big(\Pi (t) x\big)_i}{x_i}.
\end{equation}
We start by stating that all coefficients of the matrix $\Pi(t)$ grow like the Perron-Frobenius eigenvalue.
\begin{lem} \label{lem:hennion} \cite[Lemma 2.1]{HennionHerve08} Fix $N$. We have for all $t\geq 2$, 
\begin{equation} \label{eq:H1}
0 \leq \log \frac{\max_{i,j \leq N} \Pi_{i.j}(t)}{\min_{i,j\leq N} \Pi_{i.j}(t)} \leq  \log \frac{\max_{i,j\leq N} \omega_{i.j}(1)}{\min_{i,j\leq N} \omega_{i.j}(1)} +  \log \frac{\max_{i,j\leq N} \omega_{i.j}(t)}{\min_{i,j\leq N} \omega_{i.j}(t)},
\end{equation}
and for all $ y  \in (\R_+^*)^N$,
\begin{equation} \label{eq:H2}
1 \leq  \frac{\max_{i\leq N} (\Pi(t){y})_i}{\min_{i\leq N} (\Pi(t){y})_i} \leq  \frac{\max_{i,j\leq N} \omega_{i.j}(1)}{\min_{i,j\leq N} \omega_{i.j}(1)} .
\end{equation}
Moreover, 
\begin{equation} \label{eq:H3}
\sup_{t \geq 1} \big\vert 
\log \lambda_N^{\rm PF} (t) - 
\log {\|\Pi(t) {\bf 1} \|_1 }\big\vert  < \8 \quad {\rm a.s.} 
\end{equation}
\end{lem}
\begin{proof}
By positivity, we see that for all $m, m', n, n' \leq N$,
\begin{eqnarray*}
\Pi_{m,n}(t) = \sum_{k,\ell \leq N} \omega_{m,k}(1) \Pi_{k,\ell}(1,t-1) \omega_{\ell,n}(t)
&\leq& 
\frac{\max_{i,j\leq N} \omega_{i.j}(1)}{\min_{i,j\leq N} \omega_{i.j}(1)} \times \Pi_{m',n}(t)
\\
\Pi_{m,n}(t) 
&\leq& 
\frac{\max_{i,j\leq N} \omega_{i.j}(t)}{\min_{i,j\leq N} \omega_{i.j}(t)} \times \Pi_{m,n'}(t).
\end{eqnarray*}
This implies \eqref{eq:H1}--\eqref{eq:H2}.
By Perron-Frobenius theorem, $\exists x>0$ with
$\|x\|_1=1$ and $x^* \Pi(t) = \lambda_N^{\rm PF}(t) x^*$. 
Then, on the one hand,
\begin{eqnarray*}
x^* \Pi(t) {\bf 1} =  \lambda_N^{\rm PF}(t) x^* {\bf 1} =  \lambda_N^{\rm PF}(t),
\end{eqnarray*}
while we can estimate
\begin{equation}\nn
x^* \Pi(t) {\bf 1} = \sum_{i\leq N} x_i (\Pi (t){\bf 1})_i \leq
\|\Pi(t) {\bf 1}\|_\8  
\end{equation}
and
\begin{eqnarray*}
x^* \Pi(t) {\bf 1} = \sum_{i\leq N} x_i (\Pi (t){\bf 1})_i &\geq& 
\min_i(\Pi(t) {\bf 1})_i \; \|x\|_1\\
&=& 
\min_i(\Pi(t) {\bf 1})_i \\
&\stackrel{\eqref{eq:H2}}{\geq} & \frac{\min_{i,j\leq N} \omega_{i.j}(1)}{\max_{i,j\leq N} \omega_{i.j}(1)} 
\times \max_i(\Pi (t){\bf 1})_i 
\end{eqnarray*}
Then the claim follows using that $N^{-1}|y|_1 \leq |y|_\8 \leq |y|_1$ with $y=\Pi (t){\bf 1} \in \R^N$.
\end{proof}

We are now able to show the existence of the free energy and the Gaussian fluctuations of the logarithm of the partition function.

\begin{proof}[Proof of Theorem \ref{thm:gene-intro}]
For a $N\times N$-matrix $\chi$, set  $\| \chi\|_1=\sum_{i,j=1}^N |\chi_{i,j}|$.
	Since the norm $\|\cdot\|_1$ is submultiplicative, we see that the doubly indexed sequence $\log \|\Pi(s,t)\|_1 (0 \leq s \leq t)$ is  subbadditive and, by the subbadditive ergodic theorem (see e.g. the nice proof in \cite{Steele89}), it follows that $t^{-1}\log \| \Pi(t)\|_1$ converges a.s. to a limit $v_N$, and similarly for the entries $t^{-1}\log \| \Pi_{i,j}(t)\|_1$ by Lemma \ref{lem:hennion}. Moreover, from Theorem 3 in \cite{Hennion97},  a central limit theorem holds. 
It then suffices to recall that $Z_N(t,j)= ( \Pi(t)^* {\bf 1} )_j $.

Strict positivity of the variance follows from two results in \cite{Hennion97}. By Corollary 3 therein,
$\sigma_N =0$ implies that $(e^{-t v_N} \|\Pi(t)\|_1)_{t \geq 1}$ is a tight sequence in $(0,  \8)$. By Theorem 5, this is equivalent to a certain geometric property of the support of the law of $\fxi$, which is clearly not satisfied for  $\fxi$ with non-constant, i.i.d. entries.

It remains to prove \eqref{eq:vN=}. 
By Lemma \ref{lem:hennion}, we have
\begin{eqnarray} \nn
v_N &=& \lim_{t \to \8} \frac{1}{t} \log \|Z_N(t)\|_1 \\ \nn
&=& \lim_{t \to \8} \frac{1}{t}  \sum_{s=1}^t  \log \frac{\|Z_N(s)\|_1}{\|Z_N(s-1)\|_1}\\ \nn
&\stackrel{(\ref{eq:recursion-general})}{=}& \lim_{t \to \8} \frac{1}{t}  \sum_{s=1}^t  \log \| \fxi(s)^* X_N(s-1) \|_1
\\ \nn
&\stackrel{\rm ergodic\ th.}{=}& \E  \log \| \fxi(1)^* X_N(0) \|_1 ,
\end{eqnarray}
which is equal to the RHS of \eqref{eq:vN=}. 
\end{proof}


\subsection{Infinite volume measure} \label{sec:infvolmeas}

In this section we prove the existence of the infinite volume polymer  measure and of a co-variant measure.

\begin{proof}[Proof of Theorem \ref{prop:infinitepolymer}] Recall the definition \eqref{eq:polymer-measure} of the 
finite horizon P2L polymer measure. It is well known, and easily checked, that $P_{0,i;T,\star}^\omega$ is a  time-inhomogeneous Markov chain 
on $\{1,\ldots,N\}$, with 1-step transitions given for $0 \leq t <T$ by
  \begin{eqnarray}\nn
  P_{0,i;T,\star}^\omega(j_{t+1}=\ell \big \vert j_t=k) &=& \frac{\omega_{k,\ell}(t+1) Z_N(t+1,\ell;T,\star)}
   {\sum_{1\leq m \leq N} \omega_{k,m}(t+1) Z_N(t+1,m;T,\star)} \\  \nn
   &=& \frac{\omega_{k,\ell}(t+1) \frac{Z_N(t+1,\ell;T,\star) }{ \sum_{\ell'} Z_N(t+1,\ell';T,\star)}}
   {\sum_{1\leq m \leq N} \omega_{k,m}(t+1) \frac{Z_N(t+1,m;T,\star)}{ \sum_{\ell'} Z_N(t+1,\ell';T,\star)}} 
\end{eqnarray}

   But, a.s., 
   $$  \left(\frac{Z_N(t+1,\ell;T,\star) }{ \sum_{\ell'} Z_N(t+1,\ell';T,\star)}\right)_{\ell=1}^N=\Pi(t+1,T) \cdot {\bf 1} \longrightarrow \VNI(t+1) $$
   as $T \to \8$, so the above transition converges,
    $$
   P_{0,i;T,\star}^\omega(j_{t+1}=\ell \big \vert j_t=k)  \longrightarrow
   \frac{\omega_{k,\ell}(t+1)\VNI(t+1,\ell) }{\sum^N_{\ell'=1}\omega_{k,\ell'}(t+1)\VNI(t+1,\ell')} \;.
      $$
This proves that the finite horizon P2L polymer measure converges to the Markov chain $P^\omega$ given by the transition probabilities \eqref{eq:polymerinfini}. In order to obtain \eqref{eq:polymerinfini2}, one can use a shifted version of \eqref{eq:Xinfinvariant},
	$$\fxi(t+1) \cdot \VNI(t+1)= \VNI(t), $$
	to rewrite the denominator in the RHS of \eqref{eq:polymerinfini} as
	$$
	{\sum^N_{\ell'=1}\omega_{k,\ell'}(t+1)\VNI(t+1,\ell')} = {\VNI(t,k) \| \fxi(t+1) \VNI(t+1) \|_1} .
	$$

We end by proving the second part of Theorem \ref{prop:infinitepolymer}. Note that \eqref{def:co-var} writes
\begin{equation} \label{def:co-var2}
\nu_N (t, j):= \frac{\VINI(t,j) \VNI(t,j) }{ \VINI(t)^* \VNI(t) }\; ,
\end{equation}
where $\VINI(t)^* \VNI(t) = \sum_{k=1}^N \VINI(t,k) \VNI(t,k)$,
so that, with \eqref{eq:polymerinfini2},
\begin{equation}    \label{eq:like-a-rolling-stone}
\nu_N (t, k) P^\omega (j_{t+1}=\ell \big \vert j_t\!=\!k) = 
  \frac{\VINI(t,k) \times \omega_{k,\ell}(t\!+\!1) \times  \VNI(t\!+\!1,\ell)}{\| \fxi(t\!+\!1) \VNI(t\!+\!1) \|_1 \times  \big(\VINI(t)^* \VNI(t)\big) } \;.
 \end{equation}
 
Summing over $k$ and using \eqref{eq:Xinfreverse} we get
\begin{eqnarray} \label{eq:like-a-rolling-stone3}
\sum_{k=1}^N \nu_N (t, k) P^\omega (j_{t\!+\!1}\!=\!\ell \big \vert j_t=k) 
&=& 
 \frac{\VINI(t\!+\!1,\ell) \times \| \fxi(t\!+\!1)^* \VINI(t) \|_1
 \times  \VNI(t\!+\!1,\ell)}{\| \fxi(t\!+\!1) \VNI(t\!+\!1) \|_1 \times  \big( \VINI(t)^* \VNI(t) \big) } 
\end{eqnarray}
Summing now over $\ell$ we derive the identity
\begin{equation} \label{eq:identity}
 {   \| \fxi(t\!+\!1) \VNI(t\!+\!1) \|_1  \;  \big( \VINI(t)^* \VNI(t) \big) }
 =
 {\| \fxi(t\!+\!1) \VINI(t) \|_1 \;  \big( \VINI(t\!+\!1)^* \VNI(t\!+\!1)\big)
 } 
\end{equation}

Using \eqref{eq:identity} back in the RHS of \eqref{eq:like-a-rolling-stone3} we see that it is equal to $\nu_N(t+1, \ell)$, proving the claim. 
  \end{proof}

\begin{rmk} We have in fact a time-reversal property.
Define the (time-inhomogeneous) transition probability ${\overleftarrow{P}^{\omega}}$ on $\{1,2,\ldots N\}$ by
\begin{equation} \nn
\overleftarrow{P}^\omega (j_{t\!+\!1}\!=\!k \big \vert j_t=\ell) 
=
 \frac{\omega_{k,\ell}(t+1) \VINI(t,k) }{
  \VINI(t+1,\ell)
  \| \fxi(t\!+\!1)^* \VINI(t\!+\!1\|_1} 
\end{equation}
Then, by \eqref{eq:identity}, the equality \eqref{eq:like-a-rolling-stone} writes
\begin{equation}\nn
\nu_N(t,k) {P}^\omega (j_{t\!+\!1}\!=\!\ell \big \vert j_t=k) 
=
\nu_N(t+1, \ell) \overleftarrow{P}^\omega (j_{t\!+\!1}\!=\!k \big \vert j_t=\ell) ,
\end{equation}
from which stationarity follows immediately. 
The time-reversed of the chain discussed in item 2) of Theorem \ref{prop:infinitepolymer} is the chain with transitions
${\overleftarrow{P}^{\omega}}$ and starting at time $0$ from the law $\nu_N(0,\cdot)$.
\end{rmk}


\section{Exact solution for stable laws}

We consider the particular cases when  $ \omega_{i,j} (t) \eqlaw \Sa$, the stable law of index $\alpha \in (0,1)$, see \cite{BertoinS, Durrett}.  We recall that it is supported by $\R_+$ and that, for $\l>0$,
\begin{eqnarray}\label{eq:laplTa}
	\E[e^{-\lambda \Sa}] = e^{-\lambda^{\alpha}}.
\end{eqnarray}


\subsection{The random walk representation} \label{sec:RWrepsentation}

This section is devoted to the proof of Theorem \ref{thm:structure-stable} which summarizes the probabilistic structure of the model.

Let by ${\mathcal F}_t$ be the $\sigma$-field generated by the $\omegaij(s)$ for $s \leq t $ and all $i,j$.
Property  \eqref{eq:stabilite2} directly implies that for each $j$, conditionally on $\mathcal{F}_t$, $S_N(t+1,j)$ has law $\Sa$, for all $1\leq j \leq N$. 
 To deal with the $N$-vector, we fix
 $\lambda_j>0$ for $1 \leq j \leq N$ and we compute,
 \begin{eqnarray} \nonumber
\E[  e^{-\sum_{j=1}^N \la_j Z_N(t+1,j)} | {\mathcal F}_t ]&=&
 \prod_{i, j=1}^N  \E[  e^{-\la_j Z_N(t,i) \omega_{i,j}(t+1)} | {\mathcal F}_t ]
\\ \label{eq:lastarria}
&=& \exp\{ - \sum_{j=1}^N \lambda_j^\alpha  \sum_{i=1}^N  Z_N(t,i)^\alpha  \}
\qquad ({\rm by\ } \eqref{eq:laplTa}).
\end{eqnarray}
Then
\begin{eqnarray}
	\E[  e^{- \sum_{j=1}^N \la_j S_N(t+1,j)} | {\mathcal F}_t ]
&=& \exp\{ - \sum_{j-1}^N \lambda_j^\alpha \},
\end{eqnarray}
which shows that, conditionally on $\mathcal{F}_t$, $S_N(t+1,j)$ with $ 1\leq j \leq N)$ have law $\Sa$ for all $1\leq j \leq N$ and
are conditionally independent.  As a consequence, the random variables $\{S_N(t+1,j):\, t\geq 1,\, j=1,\cdots,N\}$ are i.i.d. with common law $\Sa$. This proves Proposition \ref{thm:structure-stable}, part 1.

Now, recall the identity \eqref{eq:indepsum} 
\begin{eqnarray}\label{eq:logdecomposition}
	\log Z_N(t,j) = \log S_N(t,j) + \phi_N(t-1),
\end{eqnarray}
where $\phi_N(t)$ denotes the height of the polymer at time $t$:
\begin{equation}
\label{def:phia} \phi_N(t)= \log \|Z_N(t)\|_{\alpha}.
\end{equation}
By definition of the $\alpha$-norm,
\begin{eqnarray}
\phi_N(t+1)
&=& \alpha^{-1} \log (\sum_{i=1}^N  Z(t+1,i)^\alpha) \nonumber \\
 &=& \alpha^{-1} \log \sum_{i=1}^N  \frac{Z(t+1,i)^\alpha}{\| Z_N(t)\|^{\alpha}_{\alpha}}
+ \log \|  Z_N(t)\|_{\alpha} \nonumber \\ \label{eq:Satt}
 &=&  \alpha^{-1} \log \sum_{i=1}^N  S_N(t+1,i)^{\alpha} + \phi_N(t)\\
 &=& \log \| S_N(t+1) \|_{\alpha} + \phi_N(t).
\end{eqnarray}
where $S_N(t)$ denotes the vector $(S_N(t,1),\cdots,S_N(t,N))$.
From the independence observed above, the sequence $\{\Upsilon_N(t):\, t\geq 1\}$ defined by
$$
\Upsilon_N(t)=  \log \| S_N(t) \|_{\alpha}
$$
 is i.i.d., and $(\phi_N(t))_t$ is a random walk with jumps $\Upsilon_N(\cdot)$. The identity \eqref{eq:logdecomposition}  shows that $\{\log Z_j(t):\, j=1,\cdots,N\}$ is an independent $N$-sample of the $\alpha$-stable law with an independent shift by $ \phi_N(t-1)$. The above discussion proves Theorem \ref{thm:structure-stable}, part 3 and readily implies part 5.

We now turn to the proof of part 5 and 6 in Theorem \ref{thm:structure-stable}. 
Let $S=(S_1,\cdots,S_N)$ be a vector with i.i.d. entries following the law $\Sa$ and let $X=\frac{S}{||S||_1}$.
Then,
\begin{eqnarray*}
	\fxi(0) \cdot X = \fxi(0) \cdot S
						   =\frac{\fxi(0)\, S}{|| \fxi(0) \, S||_1}
						   = \frac{\tilde{S}}{||\tilde{S}||_1},
\end{eqnarray*}
where $\tilde{S}=\frac{\fxi(0) \, S}{||\fxi(0)\, S||_{\alpha}}$. By the conditioning argument above, we see that $\tilde{S}$ has the same distribution as $S$. Hence, $X$ and $\fxi(0)\cdot X$ have the same distribution and the law of $X$ is indeed the unique invariant measure for the system.



\subsection{Asymptotic of the Lyapunov exponent for large $N$:} \label{sec:asymptLyap}

 We will prove Proposition \ref{thm:main-stable}, namely,
\begin{equation}
\nn
\alpha v_N =\log N + \log \log N + \log c_{\alpha}+ o(1) , \quad \alpha^2 \sigma^2_N = \frac{\pi^2}{3 \log N} + o(\frac{1}{\log N}),
\end{equation}
with $c_\alpha$ from \eqref{eq:calpha}. Recall from Theorem  \ref{thm:structure-stable} that
\begin{equation} \nn
v_N = \E \Upsilon_N = \E  \log \| S_N\|_{\alpha} = \alpha^{-1} \E \log \sum^N_{j=1} S^{\alpha}(j),
\end{equation}
for $S_N=(S(j);\, j=1,\cdots,N)$ an $N$-sample of $\Sa$ independent random variables, and 
\begin{eqnarray}
	\sigma^2_N = \E \left[ \left( \Upsilon_N- v_N\right)^2\right]
\end{eqnarray}

Note that $\P[S^{\alpha}_1 > x] \sim \frac{1}{x \Gamma(1-\alpha)}$ so that the $S(j)^{\alpha}$'s are in the domain of attraction of the totally asymetric stable law of index $1$ that we will denote by $\S$. This is the stable law of index $\alpha=1$, with characteristic function given for 
$u \in \R$ by
\begin{eqnarray}  \label{def:Stable1} 
\E e^{iu \S_{}}&=& 
\exp \left\{ \int_1^\8 (e^{iux}-1) \frac{dx}{x^2} + \int_0^1 (e^{iux}-1-iux) \frac{dx}{x^2}  
\right\} \\   \nonumber
&=& 
\exp \left\{ iCu - \frac{\pi}{2} |u| \big\{ 1+ i \frac{2}{\pi} {\rm sign}(u) \ln |u| \big\} \right\}
\end{eqnarray}
for some real constant $C$ defined by the above equality.
It takes real values, not only positive ones.  We can then appeal to a general fluctuation result:
\begin{prop}  \label{lem:bbb}
	Suppose $(X_i)_i$ is a family of i.i.d. positive random variables such that
	$P[X_1>x] \sim x^{-1} L(x),$
	with $L(x)$ a slowly varying function. Let 
	\begin{eqnarray}
		a_N = \inf\{x:\, P[X_1>x]\leq 1/N\} \quad \rm{and} \quad b_N = N \, \E[X_1 {\bf 1}_{X_1<a_n}].
	\end{eqnarray}
	Then,
	\begin{eqnarray}
		\frac{\sum^N_{i=1} X_i-b_N}{a_N} \cvlaw \mathcal{S}.
	\end{eqnarray}
\end{prop}
\begin{proof} 
	This is a particular case of \cite[Th. 3.7.2]{Durrett}.
\end{proof}
	
In our case, we can replace $a_N, b_N$ by their leading order (denoting them with the same symbol for simplicity): with $a_N = \frac{N}{\Gamma(1-\alpha)}$ and $b_N = \frac{N\log N}{\Gamma(1-\alpha)}$, we have
\begin{eqnarray}\label{eq:renormstable}
	\S_N:=\frac{\Gamma(1-\alpha)}{N} \times \left\{  \sum^N_{j=1} S(j)^{\alpha} - \frac{N\log N}{\Gamma(1-\alpha)}\right\} \cvlaw \S,
\end{eqnarray}
as $N\to\infty$. In particular,
\begin{eqnarray}
	\frac{\sum^N_{j=1} S(j)^{\alpha} }{N \log N}    \to \frac{1}{\Gamma(1-\alpha)},
\end{eqnarray}
in probability, as $N\to\infty$. The asymptotics for $v_N$ follows from next proposition:
\begin{lem}  \label{prop:aaa}
	\begin{eqnarray}
		\lim_{n\to \infty}\E \left[ \log \frac{\sum^N_{j=1} S(j)^{\alpha} }{N \log N} \right] =-\log \Gamma(1-\alpha).
	\end{eqnarray}
\end{lem}
\begin{proof} 
	To derive convergence of moments from the convergence in probability above, we prove uniform integrability, following arguments of \cite[proof of Prop. 2.8]{CN84}. It suffices to show that
 \begin{equation} \label{eq:d2-1}
 \sup_N \E  \big( \Sigma_N  /(N \log N)\big)^a < \8 
\end{equation}
 both for some for $a>0$ and some $a<0$, where $\Sigma_N = \sum_{i=1}^N S(i)^{\alpha}$.
 
 For the first purpose, let $\hat \Sigma_N = \sum_{i=1}^N  S_i^{\alpha} {\mathbf 1}_{\{S_i^{\alpha}  \leq Nu\}} $ and note that 
 $$\P( \Sigma_N  \neq \hat \Sigma_N ) \leq N \P(  S_i^{\alpha}  > Nu) \sim \Gamma(\alpha)\frac{\sin (\pi \alpha)}{\pi u }$$
  as $N\to\infty$. Use Markov inequality and bound, for $u>0$,
 \begin{eqnarray*}
\P( \Sigma_N  /(N \log N) > u) &\leq& u^{-1} \E \big( \hat \Sigma_N  /(N \log N) \big) + \P ( \Sigma_N  \neq \hat \Sigma_N ) \\
&=& {\mathcal O} ( u^{-1} \log u ),
\end{eqnarray*}
which implies \eqref{eq:d2-1} for $a \in (0,1)$. 

We now treat the case $a=-1$: write by Fubini's theorem
\begin{equation}
\nn
\E \Big[\big( \Sigma_N  /(N \log N)\big)^{-1} \Big]=  \int_0^\8    \E e^{-t \Sigma_N  /(N \log N) }    dt
=\int_0^\8 g\big( t  /(N \log N)\big)^N dt,
\end{equation}
where
\begin{equation}  \label{d2-2} 
 \qquad g(t) = \E \exp\{-t (\Sa)^\alpha\} 
 \leq
 \left\{
 \begin{array}{cc}
 1-C_1t| \log t| & t\in (0,1/2],   \\
 C_2 &   t \in [1/2, \log N],   \\
  t^{-\gamma},   & t > \log N,  
\end{array}
\right.
\end{equation}
for some $C_1>0, C_2 \in (0,1)$ and $\gamma>0$. The bound for $t \in (0,1/2)$ follows from the explicit rate of decay of the tails of $\Sa$, whereas the one for $t>2$ follows from the fact that
the density of $\Sa$ at $s$ is ${\mathcal O}(s)$. Plugging the bounds (\ref{d2-2}) into the above integral, we get (\ref{eq:d2-1}) for $a=-1$.
\end{proof}

This ends the proof of Proposition \ref{thm:main-stable}
and yields $\alpha v_N \sim \log N + \log \log N - \log \Gamma(1-\alpha)$. 
\vspace{2ex}

As for the variance, first note that, by \cite[Lemma 2.6]{CN84} (or even the arguments in the proof above), $\log \sum^N_{j=1} S(j)^{\alpha} $ has finite variance. Now, with $\Sigma_N$ as in the proof of the previous lemma,
\begin{eqnarray}
	\log\left( \frac{\Sigma_N}{c_{\alpha}N\log N}\right)
	= \log\left( 1+\frac{1}{\log N} \mathcal{S}_N\right) = F_N(\S_N),
\end{eqnarray}
where $\S_N$ is defined in \eqref{eq:renormstable}, $c_\a= 1/\Gamma(1-\a)$  and $F_N(x) = \log(1+\frac{x}{\log N})$. Let $L>0$ and observe that the functions $\log N \times F^2_N(x)$ are uniformly bounded and uniformly continuous on $\{x \leq L \sqrt{\log N}\}$. Assume that, on our probability space, $\S_N \to \S$ almost surely. Then,
\begin{eqnarray}
	\log N \, \E\left[ F^2_N(\S_N) {\bf 1}_{|S_N|\leq L \sqrt{\log N}} \right]
	&\sim& \log N \, \E\left[ F^2_N(\S) {\bf 1}_{|S_N|\leq L \sqrt{\log N}} \right]\\
	&\sim& \log N \int^{\infty}_0 F^2_N(y)\, \frac{dy}{y^2}.
\end{eqnarray}
By a simple change of variable, the last quantity equals $\int^{\infty}_0 \log^2(1+y)\, \frac{dy}{y^2}=\frac{\pi^2}{3}$.




\subsection{Scaling of the polymer height}\label{sec:poly-height}

 \begin{proof} Theorem \ref{thm:main-front}.
In this section, we proceed as in \cite{CQR}, Theorem 3.2.   
 
We first expand $\Upsilon_N$ from \eqref{eq:accr-stable}. Recall that
\begin{eqnarray*}
	\Upsilon_N = \frac{1}{\alpha} 	\log \sum^N_{j=1}S^{\alpha}_j,
\end{eqnarray*}
for $S_1,\cdots,S_N$ independent $\Sa$-distributed random variables.
From \eqref{eq:renormstable}, we get
\begin{eqnarray*}
	\sum^N_{j=1}S^{\alpha}_j
	=\frac{N\log N}{\Gamma(1-\alpha)}\left\{ 1 + \frac{\mathcal{S}_N}{\log N}\right\}
\end{eqnarray*}
with $\mathcal{S}_N$ converging to $\S$, the totally asymmetric stable law of exponent $1$ given by \eqref{def:Stable1} . Hence,
\begin{equation}
\nn 
\Upsilon_N = \frac{1}{\alpha} \log \left( \frac{N \log N}{\Gamma(1-\alpha)}\right) + \frac{  \S_N}{\alpha \log N} + o(1/\log N) 
\end{equation}
so that, for all $t \geq 1$, 
\begin{eqnarray}
\S_N(t) := 	\alpha \log N \left\{ \Upsilon_N(t) - \log \beta_N\right\} \to \S,
\end{eqnarray}
with $\beta_N = \left( \frac{N \log N}{\Gamma(1-\alpha)}\right)^{1/\alpha}$. This gives the scaling limit of the jumps of the walk $\phi_N$. 
To get the scaling limit for the walk itself,  we use basic convergence theorems  of independent increments processes. 
We apply Theorems 3.2 in \cite{JacodSF}, taking $Y_m^n, \gamma^n(t)$ in formula (3.1) in that paper equal $\S_N(m), k_Nt$ respectively. We derive the desired convergence \eqref{eq:ouf} after a few manipulations.  
\end{proof}


\subsection{Asymptotics of the invariant measure}\label{sec:inv-measure}

Recall the definition
\begin{eqnarray*}
	X_{N,\alpha}(t) = \frac{Z_N(t)}{||Z_N(t)||_{\alpha}}= \frac{Z_N(t)}{\left( \sum_j Z_N(t,j)^{\alpha}\right)^{1/\alpha}}.
\end{eqnarray*}
Letting $S_N(t,j):=\frac{Z_N(t,j)}{||Z_N(t-1)||_{\alpha}}$ and $S_N(t)=(S_N(t,1),\cdots,S_N(t,N))$, we get
\begin{eqnarray*}
	X_{N,\alpha}(t) = \frac{S_N(t)}{\left( \sum_j S_N(t,j)^{\alpha}\right)^{1/\alpha}},
\end{eqnarray*}
where we recall that $\{S_N(t,j):\, t\geq 1,\, 1\leq j \leq N\}$ is an i.i.d. family of $\Sa$-distributed random variables.

We borrow a (basic) version of Theorem 5.7.1 from \cite{LLR}:
\begin{lem}
	Let $(X_i)_i$ be i.i.d. random variables with common distribution function $F$ and let $u_n=u_n(\tau)$ be such that
	\begin{eqnarray}
		n [1-F(u_n(\tau))] \to \tau.
	\end{eqnarray}
	Then, the point process 
	\begin{eqnarray}
			\sum^n_{i=1} \delta_{u_n^{-1}(X_i)}
	\end{eqnarray}
	converges weakly to the Poisson point process (PPP) on $\R_+$ with intensity measure the Lebesgue measure.
\end{lem}

In our setting, $1-F(u)\sim cu^{-\alpha}$ as $u \to +\8$ with $c=\Gamma(1-\alpha)^{-1}$ so that we can choose 
\begin{eqnarray}
	u_n(\tau)=c^{1/\alpha}n^{1/\alpha}\tau^{-1/\alpha},\qquad u_n^{-1}(x)=cnx^{-\alpha}.
\end{eqnarray}
Let $(S_i)_i$ be an i.i.d. family of $\Sa$-distributed random variables. Then, according to the above result, the point process
\begin{eqnarray}
	\mathcal{P}_N:=\sum^N_{i=1}\delta_{cNS_i^{-\alpha}}
\end{eqnarray}
converges weakly to a Poisson point process $(\tau_i)_i$ with intensity $d\tau $ on $\R_+$.

A more familiar formulation is that $\{\frac{S_i}{N^{1/\alpha}}\}$ converges to a PPP with intensity proportional to $s^{-(1+\alpha)}$. Accordingly, it is tempting to rewrite
\begin{eqnarray}
	X_{N,\alpha}^\8(t) = \frac{N^{-1/\alpha}S_N(t)}{\left( \sum^N_{j=1} \left( N^{-1/\alpha}S_N(t,j)\right)^{\alpha}\right)^{1/\alpha}},
\end{eqnarray}
but the denominator  diverges, as
\begin{eqnarray}
	 \frac{\sum^N_{j=1}S_N(t,j)^\alpha}{N \log N} \to \frac{1}{\Gamma(1-\alpha)},
\end{eqnarray}
so that it is not unreasonable to expect
\begin{eqnarray}
 (\log N)^{1/\alpha} \times X_{N,\alpha}(t) = \frac{N^{-1/\alpha}S_N(t)}{\left(\frac{\sum^N_{j=1}S_N(t,j)^\alpha}{N \log N} \right)^{1/\alpha}}
\end{eqnarray}
to converge in some sense. Since the limit would have infinitely many points in the neighborhood of 0, we perform a non-linear transformation on the point process, and consider $\sigma_i= - \log \tau_i$.
\begin{prop} Assume $\omega_{i,j} \sim \Sa$. We have the convergence of the point process
\begin{equation} \label{eq:cv-pp}
\sum_{i=1}^N 	\delta_{ \alpha \log X_{N,\alpha}^\8(i)+ \log_2 N} \cvlaw
\sum_{i\geq 1} \delta_{\sigma_i}
\end{equation}
with $(\sigma_i)_i$ a Poisson point process  with intensity $e^{-\sigma} d\sigma $ on $\R$. ($\log_2=\log \log$.)
\end{prop}

We make a comment on localization: Directed polymers on the lattice in strong disorder regime have macroscopic atoms, in the sense that 
the favorite sites at the ending time have mass bounded away from 0  
\cite{CSY, Bates-Chatterjee,  Bates}. Here, the largest mass vanishes at order $(\log N)^{-1/\alpha}$.
Localization is thus rather weak in our model.


\section{Perturbative results}

As a preliminary, we show an additional result  when the environment is distributed as the stable law.

\subsection{Wave front in the $\alpha$-stable case}  \label{sec:frontSa}

Recall $U_N(t, x), u_{\alpha}$ and $\phi_N$ from \eqref{def:front}, \eqref{def:ualpha} and
\eqref{def:Phi_N} respectively.

\begin{prop} \label{prop:frontprofilestable}
	Assume $\omega_{i,j} \sim \Sa$, and fix $t\geq1$ arbitrary.	
	\begin{enumerate}
		\item \label{eq:pert1} Conditionally on $\mathcal{F}_t$, we have a.s., as $N \to \8$, 
			$$U_N\big(t, x + \phi_N(t-1)\big) \to u_{\alpha}(x),$$
			uniformly in $x$.
			
		\item  \label{eq:pert2} As $N\to \infty$,
			$$
			 \log N \times \Big[ U_N\big(t, x + (t-1)\, v_N+\phi_N(0)\big)- u_\alpha(x)\Big] \cvlaw u_{\alpha}'(x) \mathcal{Z}.$$
			where $\mathcal{Z}$ is distributed as a sum of $t-1$ independent $\mathcal{S}$ random variables
			(see definition   \eqref{def:Stable1})
			i.e., equal in law to $(t-1) \mathcal{S}$ up to a shift.
	\end{enumerate}
\end{prop}

\begin{proof}
	Let 
	$$X_N(t,j) \stackrel{\eqref{eq:logdecomposition}}{\equiv} \log S_N(t,j) := \log Z_N(t,j)-\phi_N(t-1)$$ 
	Recall from Theorem \ref{thm:structure-stable} that the random variables $\{X_N(t,j):\, t\geq 1,\, 1\leq j \leq N\}$ are i.i.d. with common law $\log \Sa$. Hence, conditionally on $\mathcal{F}_{t-1}$, $U_N(t, x + \phi_N(t-1))$ is a sum of independent Bernoulli random variables with parameter $u_{\alpha}(x)$. Pointwise convergence in item \eqref{eq:pert1} then follows from the law of large numbers. Further,  pointwise convergence of monotone functions to a continuous limit is uniform on compact by Dini's theorem; in fact, the limit is even uniform on $\R$ because the functions are bounded. All this 	yields  item \eqref{eq:pert1}.

	To prove item \eqref{eq:pert2}, first observe that
	$$U_N(t, x + (t-1)\, v_N+\phi_N(0)) 
	= \sum^N_{j=1} {\bf 1}_{ \{ X_N(t,j) > x - [\phi_N(t-1)-\phi_N(0)-(t-1)v_N] \} }$$
	is, conditionally on ${\mathcal F}_{t-1}$,  a binomial r.v..
	By \eqref{def:Phi_N} and \eqref{eq:accr-stable},
	\begin{eqnarray*}
		\alpha [\phi_N(t-1)-\phi_N(0)-(t-1)v_N]
		&=& \alpha \sum^{t-1}_{l=1} \left[ \Upsilon_N(l) - v_N\right]\\
		&=& \sum^{t-1}_{l=1} \left[ \log \| S_N(l)\|^{\alpha}_{\alpha} - \alpha v_N\right]\\
	\end{eqnarray*}
Since the i.i.d. variables $S(l,j)^\alpha$  belong to the domain of attraction of $\mathcal S$, 	
	we can write 
	$$ \|S_N(l) \|_\alpha^\alpha =\sum^N_{j=1}S(l,j)^{\alpha}=: c_{\alpha}N \log N+c_{\alpha} N \S_N(l)$$
	where the variable  $\S_N(l)$ defined by this equality is such that  $\S_N(l)\cvlaw \mathcal{S}$ as $N \to \8$ (see \eqref{eq:renormstable}). 
	Since $e^{\alpha v_N}=c_{\alpha}N \log N(1+o(1))$ by \eqref{eq:val-v_N}, we have
	\begin{eqnarray*}
		\alpha\big[\phi_N(t)-\phi_N(0)-tv_N\big]
		&=& \sum^t_{l=1} \left[ \log\left( 1+ \frac{1}{ \log N} \S_N(l) \right)+o(1) \right].
	\end{eqnarray*}
From this, we deduce that $\alpha \log N \times [\phi_N(t-1)-\phi_N(0)-(t-1)v_N] $
	converges in law to the sum of $(t-1)$ independent $\mathcal{S}$ random variables.

	Moreover, for any sequence $z_N\to 0$, the central limit theorem for binomial variables implies,
	$$N^{1/2} \times \left[ 
			\frac{1}{N}\sum^N_{j=1} {\bf 1}_{ \{ \alpha X_N(t,j) > \alpha x - z_N \} }
			- u_{\alpha}(x-z_N)
			\right]
			\cvlaw \mathcal{N}(0,u_{\alpha}(x) [1-u_{\alpha}(x)]).$$
	Together with a suitable Taylor expansion, we see that the Gaussian fluctuations vanish at the relevant scale,
$$ \log N \times \left[ 
			\frac{1}{N}\sum^N_{j=1} {\bf 1}_{ \{ \alpha X_N(t,j) > \alpha x - z_N \} }
			- u_{\alpha}(x) + u'_{\alpha}(x)\cdot z_N
			\right]
			\to 0.$$
	Taking  $z_N = \alpha [\phi_N(t-1)-\phi_N(0)-(t-1)v_N]$ and recalling the stable limit for $ z_N \log N $, we complete the proof of the 
second statement.
\end{proof}

\subsection{Wave front for perturbations of the $\alpha$-stable case} \label{sec:proof-perturb} 

We give the proof of Theorem \ref{th:perturb}. It is plain to see that hypothesis \eqref{eq:closetoSa} is equivalent to
\begin{equation}
\label{eq:closetoSa2}
\varphi(u) = \exp\{ - u^\alpha (1+\eps(u))\} \quad {\rm with} \quad \lim_{u \to 0^+} \eps(u)= 0 \;, 
\end{equation}
 by considering the function $\eps(u)= -u^{-\alpha} \log \varphi(u) -1$ for positive $u$. 
Define
\begin{equation} \label{eq:barZ}
	\bar{Z}_N(t,j):= \frac{Z_N(t,j)}{||Z_N(t-1)||_{\alpha}} = \sum^N_{i=1} a_{N,i}\,\omega_{ij}(t),
\end{equation}
where $a_{N,i}=\frac{Z_N(t-1,i)}{||Z_N(t-1)||_{\alpha}}$. Note that that $a_{N,i}$ are functions of $\fxi(1),\ldots \fxi(t-1)$ and that 
$\sum^N_{i=1}a_{N,i}^{\alpha}=1$. Now,
\begin{eqnarray*}
	\E[\exp\{-u \bar{Z}_N(t,j)\}|\mathcal{F}_{t-1}]
	&=&
	\prod^N_{i=1} \varphi(u a_{N,i}) \\
	&=&
	\exp\left\{ - \sum_{i=1}^N (u a_{N,i})^\alpha \big(1+\eps(u a_{N,i})\big)\right\} \qquad ({\rm by}\;  \eqref{eq:closetoSa2})\\
	&=& 
	\exp\left\{ - u^\alpha \left( 1 + \sum_{i=1}^N a_{N,i}^\alpha \eps(u a_{N,i})\right)\right\}.
\end{eqnarray*}
In the next section we will prove the following
\begin{lem} \label{lem:aNismall}
Under the above hypothesis \eqref{eq:closetoSa},
\begin{eqnarray} \nn
	\sup_{i\leq N}\{ \|a_{N,i}\|_\8\} \equiv	\sup\{ a_{N,i}; 1 \leq i \leq N, \fxi(1),\ldots, \fxi(t-1)\}  = o(1) 
\end{eqnarray}
in probability as $N \to \8$ when $t \geq 2$.
\end{lem}
With the lemma at hand, we finish the proof. Then, the computation above yields
$$
\E[\exp\{-u \bar{Z}_N(t,j)\}|\mathcal{F}_{t-1}] = \exp \left\{ - u^\alpha \big( 1 + o(1) \big)\right\}
\longrightarrow \exp \left\{ - u^\alpha\right\} = E[ \exp\{ - u \Sa\}]
$$
for all $u\geq 0$. This implies (e.g., \cite{FellerII} theorem 2, p.431) that $\bar{Z}_N(t,j) \cvlaw \Sa$, which is the first claim of the theorem.
Conditionally on ${\mathcal F}_{t-1}$, the variables $({Z_N(t,i)}/{||Z_N(t-1)||_{\alpha}}:\, i \in K_N)$ are i.i.d., so the second claim follows directly. 

We end by proving item (2). From the law of large numbers for triangular arrays of i.i.d. random variables, the claim follows for all real $x$. 
Then, uniformity is obtained similarly to the proof of Proposition \ref{prop:frontprofilestable}. This ends the proof of Theorem \ref{th:perturb}. \qed 
\subsection{Proof of Lemma \ref{lem:aNismall}}
We start with a simple proof under a stronger assumption, because the coupling argument there makes things transparent. In a second step, we give the proof under our general assumption.

\subsubsection{Coupling}
In this section we present the proof  under more restrictive hypothesis: We assume there exist constants $a<b$ such that
\begin{eqnarray}
	u_\alpha(x-b) \leq \P\big( \log \omega_{i,j}(t) \leq x \big) \leq u_\alpha(x-a) \;, \quad x \in \R.
\end{eqnarray}
with $u_{\alpha}$ from \eqref{def:ualpha}.
This allows to couple the environment $\omega$ at time $t-1$ with i.i.d. $\alpha$-stable random variables $s_{i,j}$ on a larger probability space in such a way that
\begin{eqnarray}
		c^{-1} s_{ij} \leq \omega_{ij}(t-1)\leq c s_{ij},\quad 1\leq i,j \leq N,
\end{eqnarray}
with $c=\max\{a^{-1}, b\}>1$. Define
	\begin{eqnarray}
		\tilde{Z}_N(t-1,j) = \sum^N_{i=1} Z_N(t-2,i) s_{ij}
	\end{eqnarray}
	Then, for some constant $C$, we have
	\begin{eqnarray}
		a_{N,i} \leq C \frac{\tilde{Z}_N(t-1,i)}{|| \tilde{Z}_N(t-1) ||_{\alpha}} = C \frac{S_i}{\left( \sum^N_{j=1} S_j^{\alpha}\right)^{1/\alpha}},
	\end{eqnarray}
	where the random variables $S_j = \frac{\tilde{Z}(t-1,j)}{||\tilde{Z}(t-1)||_{\alpha}},\, j=1,\cdots, N$ are now independent and $\Sa$-distributed. As noted above, the $S^{\alpha}_j$'s are in the domain of attraction of $\mathcal{S}$, so that the denominator is $O((N\log N)^{1/\alpha})$. Recall that the sum of the $S_i$'s is $O(N^{1/\alpha})$, we conclude that 
	$\sum^N_{i=1} a_{N,i} = O\left( \frac{1}{(\log N)^{1/\alpha}}\right)$ in probability. The estimate is uniform on ${\mathcal F}_{t-1}$. 


\subsubsection{Proof completed}
In this section we present the proof  under  the standing hypothesis  \eqref{eq:closetoSa}.

First, recall a well-known relation on tails of distribution and Laplace transform of a positive r.v. $V$: 
The equivalent hypothesis  \eqref{eq:closetoSa2},  $\E e^{-uV} = \exp\{ - u^\alpha (1+\eps(u))\}$ with $ \lim_{u \to 0^+} \eps(u)= 0$,
is itself equivalent,  by the Tauberian theorem (Corollary 8.1.7 in \cite{BinghamGoldieTeugels})
to
\begin{equation} \label{eq:tails}
\P[V \geq x] \sim \frac{1}{x^{\alpha}\Gamma(1-\alpha)}, \quad x \to \infty.
\end{equation}
We note that $\bar Z_{N}(t,i)$ defined in \eqref{eq:barZ},  satisfies \eqref{eq:closetoSa2} (equivalently, \eqref{eq:closetoSa}).
Indeed,
$$ \E(\exp\{ - u \bar Z_{N}(t,j) \} | {\mathcal F}_{t-1}) = \exp \big\{ - u^\alpha \big(1+ \sum_{i=1}^N a_{N,i}^\alpha  \eps(a_{N,i}u)\big) \big\},$$
where we can bound the sum uniformly by $\max\{ \eps(v); v \in [0,u]\}$, which vanishes as $u \to 0$. Then, we apply the above to 
$V= \bar Z_{N}(t-1,i)$, which obeys the tail estimate \eqref{eq:tails}. 

Thus, we can apply the theory of extreme statistics for triangular arrays of i.i.d. real variables, which implies that the 
maximum of   $\bar Z_{N}(t-1,i)$ over $i$ normalized by $N^{-1/\alpha}$
converges to a Frechet law. In complete details, we have (e.g., \cite{Bovier}, Th. 2.28)
\footnote{  
     \eqref{eq:tails} implies that the solution $v_N(\tau)$ of  $ N \P(  \bar Z_{N}(t-1,j) \geq v_N(\tau)|{\mathcal F}_{t-2}) = \tau$ is such that
           $$ v_N(\tau) \sim \frac{N}{\Gamma(1-\alpha)u^\alpha}$$
     uniformly over ${\mathcal F}_{t-2}$. This is all what we need in the i.i.d. case to apply the theorem.
}
\begin{equation}\nn
\left\{ \frac{\bar Z_N(t-1,i)}{N^{1/\alpha}}:i\leq N\right\} \cvlaw {\rm PPP\ with\ intensity\ } \frac{\alpha}{\Gamma(1-\alpha)} x^{-(1+\alpha)} {\bf 1}_{\R_+} .
\end{equation}
This point process has a finite right-most point and the sum of the $\alpha$-powers of its terms diverges. Hence,
\begin{eqnarray}
	\log a_{N,i} &\leq& 
	\max_{j\leq N} \log \frac{\bar Z_N(t-1,j)}{N^{1/\alpha}} - \frac{1}{\alpha} \log \sum_j \left( \frac{\bar Z_N(t-1,i)}{N^{1/\alpha}}\right)^{\alpha}
	\to -\infty,\quad N\to\infty.
\end{eqnarray}
This proves the lemma. \qed

\vspace{2ex}


\appendix

\section{Projective space and stochastic contractivity}

For the sake of completeness, we give a proof of Lemma \ref{th:cvdir} and expose the principal ideas leading to it.

\subsection{Contraction in the projective space} \label{sec:contraction}
When studying random matrix products, it is convenient to introduce
the projective action of these matrices. 
See, e.g., \cite{BouLa85} for invertible matrices, and \cite{Hennion97} for positive 
matrices. 
Projectively, that is, when only the directions are considered, the elements
of the positive $N$-dimensional orthant  are represented by points of the open and closed ÒpolygonsÓ
\begin{equation}
\label{def:B}
B=\{x \in (\R^*_+)^N; \|x\|_1=1\}, \qquad \bar B=\{x \in (\R_+)^N; \|x\|_1=1\}.
\end{equation}
If $g$ is a $N\times N$ matrix with (strictly) positive entries, we denote the projective action of
$g$ on $B$ by the notation $\cdot$,
$$
g \cdot x = \frac{g x}{\|g x\|_1}, 
$$
which belongs to $B$ for all $x \in \bar B$.  
The set $B$ can be equipped with a convenient metrics $d(x,y)$, which relates to the Hilbert distance (see \cite[Remark below Prop. 3.1]{Hennion97}). For $x,y \in \bar B$ let
$$
m(x,y)= \sup\{ \lambda \geq 0: \lambda y_i \leq x_i, \forall i \leq N\} \in [0,1] ,
$$
and define, with $\phi(s)=(1-s)/(1+s)$,
$$
d(x,y)= \phi\big( m(x,y)m(y,x)\big) .
$$
\begin{prop}[Sect. 10 in \cite{Hennion97}] \label{prop:distance}
The map $d: \bar B \times \bar B \to [0,1]$ defines a distance on  $\bar B$, and we have:
\medskip

(i) The topology of $(B,d)$ is the topology on $B$ induced by the usual topology on $\R^N$.
For $x \in \bar B \setminus B$ and $y \in B$, we have $m(x,y)=0$ and then $d(x,y)=1$. Also,
$$ \|x-y\|_2 \leq 2 d(x,y), \quad x,y \in \bar B.$$
Define the contraction coefficient $c(g)$ of a positive $N\times N$ matrix $g$ as
$$
c(g)= \sup\{ d(g\cdot x, g\cdot y); x, y \in \bar B\} \in [0,1] .
$$
Then,
\medskip

(ii) $\forall x, y \in \bar B, d(g\cdot x, g\cdot y) \leq c(g) d(x,y)$ ;\\

(iii) If $g$ has strictly positive entries, $c(g) <1$ ; \\

(iv) For the product of two positive    matrices $g, g',$ $c(g g') \leq c(g) c(g')$ ;\\

(v) $c(g^*)=c(g)$.
\end{prop}

\subsection{Stochastic contractivity} \label{app:contr}

We present the proof of Lemma \ref{th:cvdir}. Recall our principal task consists in showing the existence of a $B$-valued r.v. $\VNI$ such that,   $\Pi(t) \cdot x \to \VNI$ a.s. for all $x \in \bar B$ as $t \to \8$. 

The random sequence $c(\Pi(t))$ decreases and hence has a limit for all $\omega$.  
By item (iv) of Proposition \ref{prop:distance}, $c(\Pi(t)) \leq \prod_{s=0}^{t-1} c(\fxi(s))$.
By item (iii), each term is strictly less than 1, and by independence, we have a.s.
$$
\limsup_{t \to \8} t^{-1} \log c(\Pi(t)) \leq \E \log c(\fxi) <0 .
$$ 
(In fact, by subadditivity, the $\limsup$ is an a.s. limit, and the value is 
$\inf_t t^{-1} \E \log c(\Pi(t))$.)
Then, the random polygons $K(t, \omega)= \{ \Pi(t) \cdot x; x \in \bar B\}$ form a decreasing sequence of compact subsets of $B$, so that they have a limit, 
$$
K(\omega) = \bigcap_{t \geq 1} K(t) \neq \emptyset .
$$
If $x, y \in K(\omega)$, we have for all $t$,
$$
d(x,y) \leq \sup \{ d(\Pi(t) \cdot x', \Pi(t) \cdot y'); x', y' \in \bar B\} \leq c(\Pi(t)) ,
$$
so the distance is 0. Finally, $K(\omega)$ reduces to a random point, say $\VNI$ of $B$. 
In particular, $\VNI \in K(t)$ implies that 
$$d(\Pi(t) \cdot x, \VNI) \leq c(\Pi(t)) \to 0 ,$$
so that $\Pi(t) \cdot x \to \VNI$ in the $d$-distance as well as in the Euclidean distance
by item (i) of Proposition \ref{prop:distance}.

\qed


\begin{thebibliography}{}

\bibitem{Arnold}
L. Arnold, V. Gundlach, L. Demetrius:
Evolutionary formalism for products of positive random matrices.
{\it Ann. Appl. Probab.} {\bf 4} (1994),  859--901.


\bibitem{Bakhtin-Khanin}
Y. Bakhtin, K. Khanin: Localization and Perron--Frobenius theory for directed polymers 
{\it Mosc. Math. J.} {\bf 10} (2010),  667--686

\bibitem{Bates}
E. Bates:
Localization of directed polymers with general reference walk.
{\tt arXiv:1708.03713
} preprint 2017. 



\bibitem{Bates-Chatterjee}
E. Bates, S. Chatterjee:
The endpoint distribution of directed polymers.
{\tt arXiv:1612.03443} preprint 2016. 

\bibitem{BinghamGoldieTeugels}
N. Bingham, C. Goldie, J. Teugels: {\it Regular variation}. Encyclopedia of Mathematics and its Applications, 27. Cambridge University Press, Cambridge, 1987. 

\bibitem{BertoinS}
J.~Bertoin: {Subordinators: examples and applications.} Lectures on probability theory and statistics (Saint-Flour, 1997), pp.1--91, Lecture Notes in Math., 1717, Springer, Berlin, 1999.

\bibitem{Bolthausen}
E. Bolthausen: { A note on the diffusion of directed polymers in a random environment}, {\it Comm. Math. Phys.} {\bf 123} (1989), no. 4, 529--534

\bibitem{BouLa85}
P.~Bougerol, J.~Lacroix: {\it Products of random matrices with applications to Schršdinger operators.} Progress in Probability and Statistics, 8. Birkh\"auser Boston,  1985.

\bibitem{Bovier} 
A.~ Bovier: {\it Gaussian processes on trees. From spin glasses to branching Brownian motion. }
Cambridge Studies in Advanced Mathematics, 163. Cambridge University Press, Cambridge, 2017

\bibitem{BrunetDerrida00}
E.~Brunet, B.~Derrida: 
Probability distribution of the free energy of a directed polymer in a random medium. 
{\it Phys. Rev. E} (3) {\bf 61} (2000), no. 6, part B, 6789--6801

\bibitem{BrunetDerrida04}
E. Brunet, B. Derrida
{  Exactly soluble noisy traveling-wave equation appearing in the problem of di- rected polymers in a random medium.} {\it Phys. Rev. E} 70, 016106 (2004)

\bibitem{Buffet-Pule}
E. Buffet, J. Pul\'e, { Directed polymers on trees: a martingale approach}, {\it J. Phys. A: Math. Gen.} {\bf  26} (1993) 1823--1834

\bibitem{CaGuHuMe}
P.~Carmona, F.~Guerra, Y.~Hu, O.~Mejane:
Strong disorder for a certain class of directed polymers in a random environment. 
{\it J. Theoret. Probab.} {\bf 19} (2006), 134--151

\bibitem{CN84}
J.~Cohen, C.~Newman: The stability of large random matrices and their products. 
{\it Ann. Probab. } {\bf 12} (1984), 283--310. 

\bibitem{SF}
F.~Comets: {\it Directed polymers in random environments.}  Lecture Notes in Mathematics, 2175. Springer,  2017. 

\bibitem{CQR}
F.~Comets, J.~Quastel; A.~Ram\'irez:
Last passage percolation and traveling fronts. 
{\it J. Stat. Phys.} {\bf 152} (2013),  419--451

\bibitem{CSY}
F. Comets, T. Shiga, N. Yoshida:
Directed polymers in a random environment: path localization and strong disorder.
{\it Bernoulli}
{\bf 9}  (2003), 705--723.

\bibitem{Cook-Derrida}
J.~Cook, B.~Derrida: Directed polymers in a random medium: 1/d expansion and the n-tree approximation
{\it J. Phys. A: Math. Gen.} {\bf 23} (1990), 1523--1554  


\bibitem{Cortines}
A.~Cortines:
The genealogy of a solvable population model under selection with dynamics related to directed polymers. 
{\it Bernoulli} {\bf 22} (2016), 2209--2236.


\bibitem{DerridaSpohn}
B.~Derrida, H.~Spohn:
Polymers on disordered trees, spin glasses, and traveling waves. 
{\it J. Statist. Phys. } {\bf 51} (1988), 817--840. 

\bibitem{Durrett}
R. Durrett:
{\it Probability: theory and examples}. Fourth edition. Cambridge University Press, Cambridge, 2010.

\bibitem{Eckmann89}
J.-P. Eckmann, C. Waynes:
The largest Lyapunov exponent for random matrices and directed polymers in a random environment. 
{\it Commun. Math. Phys.} {\bf 121}  (1989) 147--175.

\bibitem{FellerII}
W. Feller:
{\it An introduction to probability theory and its applications.} Vol. II. 
2nd ed.,  John Wiley \& Sons, 1971.


\bibitem{Furst63}
H.~Furstenberg:  Noncommuting random products. {\it Trans. Amer. Math. Soc.} {\bf 108} (1963) 377--428

\bibitem{FurstKest60}
H.~Furstenberg, H.~Kesten: 
Products of random matrices. {\it Ann. Math. Statist.} {\bf 31} (1960)  457--469. 

\bibitem{GuivRaugi85}
Y.~Guivarc'h, A.~Raugi: 
 Fronti\`ere de Furstenberg, propri\'et\'es de contraction et th\'eor\`emes de convergence. {\it Z. Wahrsch. Verw. Gebiete} {\bf 69} (1985) 187--242.
 
 \bibitem{Henley-Huse}
 D. Henley, C. Huse:
{ Pinning and roughening of domain wall in Ising systems due to random impurities}, {\it Phys. Rev. Lett.} {\bf 54} (1985), 2708--2711.

\bibitem{Hennion97}
H. Hennion: Limit theorems for products of positive random matrices. 
{\it Ann. Probab.} {\bf 25} (1997) 1545--1587.

\bibitem{HennionHerve08}
H. Hennion, L. Herv\'e:
 Stable laws and products of positive random matrices. {\it J. Theoret. Probab.} {\bf  21} (2008) 966--981.
 
 
 \bibitem{Imbrie-Spencer}
 J. Imbrie, T. Spencer:
{ Diffusion of directed polymers in a random environment}
{\it J. Stat. Phys.} {\bf 52} 609 (1988) 609-626

\bibitem{JacodSF}
J.~Jacod: {\it Th\'eor\`emes limite pour les processus}. Saint-Flour notes. Lecture Notes in Mathematics,
vol. {\bf 1117}, pp. 298--409. Springer (1983)

\bibitem{KPZ}
M. Kardar, G. Parisi, Y. Zhang:
{ Dynamic Scaling of Growing Interfaces}
{\it Phys. Rev. Lett. } {\bf 56}, 889 (1986), 889-892

\bibitem{KZ}
 M. Kardar, Y. Zhang:
 { Scaling of Directed Polymers in Random Media}
{\it Phys. Rev. Lett.} {\bf 58}, 2087 (1987) , 2087-2090

 \bibitem{LLR}
M. Leadbetter, G. Lindgren, H. Rootz\'en: {\it Extremes and related properties of random sequences and processes}. 
Springer Series in Statistics. Springer-Verlag, New York-Berlin, 1983 

\bibitem{Newman86}
C. Newman: The distribution of Lyapunov exponents: exact results for random matrices. 
{\it Comm. Math. Phys.} {\bf 103} (1986),  121--126

 
 \bibitem{OCY}
N. O'Connell, M. Yor:
{Brownian analogues of Burke's theorem}, {\it Stoch. Process.
Appl. } {\bf 96} (2001), 285--304.
 
 \bibitem{Sinai}
 Ya. Sina\"i: A random walk with a random potential. 
 {\it Teor.
Veroyatnost. i Primenen.} {\bf 38} (1993) 457--460.

 
 \bibitem{Steele89}
 J.M. Steele:
  Kingman's subadditive ergodic theorem. 
  {\it Ann. Inst. H. Poincar\'e Probab. Statist.} {\bf 25} (1989),  93--98.


\bibitem{Timo}
T. Sepp\"al\"ainen:
 Scaling for a one-dimensional directed polymer with boundary conditions.
{\it Ann. Probab.}
{\bf 40} (2012), 19-73.

 \end{thebibliography}
\end{document}